\documentclass[square]{imsart}

\RequirePackage[OT1]{fontenc}
\RequirePackage[numbers]{natbib}
\RequirePackage[colorlinks,citecolor=blue,urlcolor=blue]{hyperref}
\RequirePackage{hypernat}

\makeatletter \@addtoreset{equation}{section} \makeatother
\usepackage{graphicx}
\usepackage{amsmath}
\usepackage{amssymb}
\usepackage{amsfonts}
\usepackage{amsthm}
\usepackage{thmtools}
\usepackage{caption}
\usepackage{subcaption}

\usepackage{cleveref}
\usepackage[noend]{algpseudocode}
\usepackage{algorithm}
\usepackage[leftcaption]{sidecap}

\newtheorem{lem}{Lemma}
\newtheorem{thm}{Theorem}
\newtheorem{ehmthm}{Hypergeometric Representation Theorem}

\newtheorem{coro}{Corollary}

\newtheorem{bibn}{Binomial Bound}
\newtheorem{hypbn}{Hypergeometric Bound}

\newtheorem{serfbnd}{Finite Sampling Bound}
\newtheorem{defn}{Definition}

\advance \topmargin by -\headheight
\advance \topmargin by -\headsep
\advance \topmargin .2in
\begin{document}

\begin{frontmatter}
\title{Exponential Bounds for the Hypergeometric Distribution}
\runtitle{Sampling without replacement: Exponential Bounds}

\begin{aug}
\author{\fnms{Evan} \snm{Greene}\thanksref{t1}\ead[label=e1]{egreene@uw.edu}}
\and
\author{\fnms{Jon A.} \snm{Wellner}\thanksref{t2}\ead[label=e2]{jaw@stat.washington.edu}}
\ead[label=u1,url]{http://www.stat.washington.edu/jaw/}

\thankstext{t1}{Supported in part by NSF Grant DMS-1104832} 
\thankstext{t2}{Supported in part by NSF Grant DMS-1104832 and NI-AID grant 2R01 AI291968-04} 
\runauthor{Greene and Wellner}


\address{Department of Statistics \\University of Washington\\Seattle, WA 98195-4322\\
\printead{e1}}

\address{Department of Statistics, Box 354322\\University of Washington\\Seattle, WA  98195-4322\\
\printead{e2}\\
\printead{u1}}
\end{aug}

\begin{abstract}
We establish exponential bounds for the hypergeometric distribution
which include a finite sampling correction factor, but are otherwise analogous 
to bounds for the binomial distribution due to Le\'on and Perron (2003) and Talagrand (1994).
We also extend a convex ordering of Kemperman's (1973) for sampling without replacement from populations of 
real numbers between zero and one:  a population of all zeros or ones (and hence yielding 
a hypergeometric distribution in the upper bound) gives the extreme case.
\end{abstract}



\end{frontmatter}


\bigskip

\bigskip

\section{Introduction and overview}
\label{sec:intro}

In this paper we derive several exponential bounds for the tail of the hypergeometric distribution. This distribution emerges as an extreme 
case in the setting of sampling without replacement from a finite population. We begin with a description of this setting.
Consider a population $C$ containing $N$ elements, $C := \left\{c_1,\dots,c_N\right\}$, with $c_i \in \mathbb{R}$. Let $N = |C|$ denote
the cardinality of this set, $a$ the value of the minimum element, $b$ the value of the maximum element, and $\mu := (N^{-1})(\sum_{i=1}^Nc_i)$, 
the population mean. 
Let $1 \leq i\leq n \leq N$, and $X_i$ denote the $i^{th}$ draw without replacement from this population. 
Finally, let $S_n := \sum_{i=1}^nX_i$ denote the sum of this sampling procedure, and let $\bar{X}_n := S_n/n$ denote the sample mean. 

R. J. Serfling 
obtained the following bound. 

\begin{serfbnd}
(\textbf{Serfling} \citep{MR0420967}) For $1 \leq n \leq N$, $S_n$ the sum in sampling without replacement, and $\lambda > 0$:
\begin{align}
P\left(\sqrt{n}(\bar{X}_n - \mu) \geq \lambda \right) \leq
\exp\left(-\frac{2 \lambda^2}{(1-f_n^{*})(b-a)^2}\right) \label{serfbnd}
\end{align}
where $f_n^{*} := (n-1)/N$. 
\end{serfbnd}

This result applies to sampling without replacement from any finite bounded population. 
Let $D, N \in \mathbb{N}$ such that $D < N$.
Then as a special case we may apply the bound to a population of $N$ elements containing $D$ $1$'s and $N-D$ $0$'s. 
Note that in this specific case $S_n =: S_{n,D,N} \sim \text{Hypergeometric}(n,D,N)$. 

For the hypergeometric distribution, the following facts are well known:
\begin{eqnarray}
P(S_{n} = k) & = & \frac{{D \choose k}{N-D \choose n-k}}{{N \choose n}}, \ \ \ \max\{ 0, n-(N-D)\} \le k \le \min\{ D , n \} , \nonumber \\ 
E(S_{n}) & = & n\left(\frac{D}{N}\right),   \nonumber \\
Var(S_{n}) & = & n \left(\frac{D}{N}\right)\left(1-\frac{D}{N}\right) \left(1-\frac{n-1}{N-1}\right) =: n \mu_{D,N}(1-\mu_{D,N})(1-f_n) \label{HG:VARIANCE}
\end{eqnarray}
with the final line defining $\mu_{D,N} := D/N$ and $f_n := (n-1)/(N-1)$.
Applying Serfling's result to the case of the hypergeometric distribution
immediately gives 
\begin{align}
P\left(\sqrt{n}(\bar{X}_{n} - \mu_{D,N}) \geq \lambda \right) \leq \exp\left(-\frac{2 \lambda^2}{(1-f_n^{*})}\right)
\label{serfling:hypergeometric:spec}
\end{align}
since $(b-a)^2=(1-0)^2=1$. 
 Comparison of the factor $1-f_n^{*}$ in Serfling's bound to the factor $1-f_n$ in \eqref{HG:VARIANCE} suggests
the following question: can Serfling's bound be improved to 
\begin{align}
P\left(\sqrt{n}(\bar{X}_n - \mu) \geq \lambda \right) \leq
\exp\left(-\frac{2 \lambda^2}{(1-f_n)(b-a)^2}\right)
\label{impserfbnd}
\end{align}
in general, or at least in the special case of the hypergeometric distribution? 


To date, the improvement conjectured in \eqref{impserfbnd} has not been obtained. For the special case of the hypergeometric,
Hush and Scovel derived the following bound by extending an argument given by Vapnik. 
See \citep{MR2206293} and \citep{MR1641250}.
\begin{hypbn}
(\textbf{Hush and Scovel}, \citep{MR2206293}) Suppose $S_n \sim \text{Hypergeometric(n,D,N)}$. Then for all $\lambda > 0$ we have
\begin{align}
P\left(\sqrt{n}(\bar{X}_{n} - \mu_{D,N})\geq \lambda\right) \leq \exp\left(-2\alpha_{n,D,N}(n \lambda^2-1)\right) \label{hushscovelbound}
\end{align}
where
\[
\alpha_{n,D,N} := \left(\frac{1}{n+1}+\frac{1}{N-n+1}\right)\vee\left(\frac{1}{D+1}+\frac{1}{N-D+1}\right)\ .
\]
\end{hypbn}

More recently, Bardenet and Maillard have improved a deficiency in Serfling's inequality that occurs when more than half
the population is sampled without replacement by using a reverse-martingale argument. 
The statement here is a specialization of their Theorem 2.4
to the hypergeometric case. See \citep{bardenet2015concentration} for additional discussion.
\begin{hypbn}
(\textbf{Bardenet and Maillard} \citep{bardenet2015concentration}) 
Suppose $S_n \sim \text{Hypergeometric(n,D,N)}$. Then for all $\lambda > 0$  and $n < N$ we have
\[
P\left(\sqrt{n}(\bar{X}_{n} - \mu_{D,N}) \geq \lambda \right) \leq \exp\left(-\frac{2 \lambda^2}{(1-n/N)(1+1/n)}\right)\ .
\]
\end{hypbn}

We will 
justify the special consideration given to the hypergeometric distribution relative to the goal of obtaining \eqref{impserfbnd}
by adapting a result of Kemperman \citep{kemperman1973moment} to derive 
a convex order between samples without replacement from populations consisting of elements in $[0,1]$ and the
hypergeometric distribution. We will then demonstrate how one may
use this convex order to  obtain exponential bounds for the more general problem of sampling without replacement from a bounded, finite population.
In doing so, we return to the setting of Serfling. In this setting we will consider the variance of the population as well. 
Anticipating this, we conclude the introduction
with a specialization of a bound of Bardenet and Maillard which incorporates information about the population variance into the bound. 
\begin{serfbnd}
(\textbf{Bardenet and Maillard} \citep{bardenet2015concentration})
For $1 \leq n < N$, $S_n$ the sum in sampling without replacement from a population $\bold{c}:=\left\{c_1,\dots,c_N\right\}$, 
$\delta \in [0,1]$, and $\lambda > 0$, we have
\begin{align}
P\left(\sqrt{n}(\bar{X}_n - \mu) \geq \lambda \right) \leq \exp\left(-\frac{\lambda^2}{2(\gamma^2+(2/3)(b-a)(\lambda/\sqrt{n}))}\right)+\delta 
\label{expbnd:bm2}
\end{align}
where 
\begin{align}
a &:= \min_{1\leq i \leq N} c_i\ ,\ b := \max_{1\leq i \leq N} c_i\ ,\ f_{n}^* := (n-1)/N\ ,\  
\mu := (1/N)\sum_{i=1}^N c_i\ ,\ \sigma^2 := (1/N) \sum_{i=1}^N (c_i-\mu)^2\ , \nonumber \\
\gamma^2 &:= (1-f_{n}^*)\sigma^2+f_{n}^*c_{n-1}(\delta)\ , 
\text{ and } \ c_n(\delta) := \sigma(b-a)\sqrt{\frac{2\log(1/\delta)}{n}}\ . \nonumber
\end{align}
\end{serfbnd}
\section{Exponential Bounds}
\label{sec:ExpBounds}

Binomial distributions arise when sampling with replacement from a population consisting only of $0$'s and $1$'s. 
As we saw in the introduction, hypergeometric distributions arise when sampling without replacement from such populations. 
Intuitively, sampling without replacement is more informative than sampling with replacement: 
when items are not replaced, eventually, when $n=N$, the entire population is sampled.
This being the case, it is natural to guess that upper bounds which apply to binomial tail probabilities will also apply to the hypergeometric tail probabilities. 

Hoeffding \citep{HOEFFDING1} proved that this guess is true for exponential bounds derived via the Cram{\'{e}}r - Chernoff method. 
This is because a convex order exists between samples with and without replacement (Hoeffding proves this order in his Theorem 4). 
Convex orders between a variety of sampling plans were subsequently explored by Kemperman \citep{kemperman1973moment} and Karlin 
\citep{karlin1974inequalities}.  


Note that by Ehm (Theorem 2 \citep{MR1093412}; see also Holmes, Theorem 3.2 \citep{MR2118602}), 
the total variation distance between the hypergeometric distribution $P_{n,D,N}^{hyper}$ 
and the binomial distribution $P_{n,D/N}^{bin}$ satisfies 
$$
d_{TV} (P_{n,D,N}^{hyper}, P_{n,D/N}^{bin}) \le \frac{n}{n+1} (1 -(D/N)^{n+1} - (1-D/N)^{n+1} )\frac{n-1}{N-1} \le \frac{n-1}{N-1},
$$
so we expect that the binomial bounds will be essentially optimal when $(n-1)/(N-1) \rightarrow 0$.  

Here we are interested in sampling scenarios in which $(n-1)/(N-1) \not\rightarrow 0$. Given the similarity in scenarios
that produce binomial and hypergeometric probabilities, one might expect that the binomial exponential bounds provide a clue to the form that hypergeometric exponential bounds will take: 
hypergeometric bounds look like binomial bounds, with a finite sampling correction factor included. 
Indeed, this is the case when we compare the bound of Serfling \eqref{serfbnd} to Hoeffding's uniform bound  
(Theorem 2.1 \citep{HOEFFDING1}), since the only difference between the two is the quantity $1-f_n^*$.
We therefore state several exponential bounds which apply to the binomial distribution. 

\begin{bibn}
(\textbf{Le{\'{o}}n and Perron} \citep{LEONPERRON1})
Let $n \geq 1$, $p\in (0,1)$, $\lambda < \sqrt{n}/2$, and $X_1,\dots,X_n$ be independent Bernoulli$(p)$ random variables.
Then
\begin{align}
P(\sqrt{n}(\bar{X}_n - p) \geq \lambda) 
\leq
\sqrt{\frac{1}{2\pi\lambda^2}}  
\left(\frac{1}{2}\right)\sqrt{\frac{\sqrt{n}+2\lambda}{\sqrt{n}-2\lambda}}e^{-\left(2\lambda^2\right)}\ .  \label{LEON_PERRON_BINOMIAL_BOUND}
\end{align}
\end{bibn}
The second bound was established by Talagrand 
\citep[pp.~48--50]{TALAGRAND1}. 
The statement here is taken from van der Vaart and Wellner 
\citep[pp.~460--462]{MR1385671}):
\begin{bibn}
(\textbf{Talagrand} \citep{TALAGRAND1})
Fix $p_0$ and consider $p$ such that $0 < p_0 \leq p \leq 1-p_0 < 1$.
Suppose for $n \in \mathbb{N}$ that $X_1,\dots,X_n$ are i.i.d. Bernoulli$(p)$ random variables.
Then there exist constants $K_1$ and $K_2$ depending only on $p_0$, such that: 
\begin{align}
&(i)\ \text{For all } \lambda > 0,\ \ \ P\left(\sqrt{n}(\bar{X}_n-p) = \lambda\right) \leq \frac{K_1}{\sqrt{n}} \exp\left(-\left[2\lambda^2 + \frac{\lambda^4}{4n}\right]\right)\ . \nonumber \\
\nonumber \\
&(ii)\ \text{For all } 0 < t < \lambda, \ \ \ P\left(\sqrt{n}(\bar{X}_n-p) \geq t\right) \leq \frac{K_2}{\lambda} \exp\left(-\left[2\lambda^2 + \frac{\lambda^4}{4n}\right]\right)\exp\left(5\lambda[\lambda-t]\right)\ .\nonumber \\
\nonumber \\
&(iii)\ \text{For all } \lambda > 0,\ \ \ P\left(\sqrt{n}(\bar{X}_n-p) \geq \lambda\right) \leq \frac{K_2}{\lambda} \exp\left(-\left[2\lambda^2 + \frac{\lambda^4}{4n}\right]\right)\ . \label{TAL:BIN:BOUND}
\end{align}
\end{bibn}
Another well-known exponential bound which applies to sums of independent random variables 
(and consequently the binomial distribution) was discovered by Bennett \cite{bennett1962probability}. 
Bennett's bound incorporates information about the population variance, and so obtains notable improvements when the population variance is small.
This statement of Bennett's inequality specialized to the binomial setting is adapted from Shorack and Wellner \cite{MR838963}. 

\begin{bibn}
(\textbf{Bennett} \cite{bennett1962probability})
Let $X_1,\dots, X_n$  i.i.d. $\text{Bernoulli}(\mu)$, with $0 < \mu \leq 1/2$. Then for all $\lambda \geq 0$
\begin{align}
P(\sqrt{n}(\bar{X}_n-\mu) \geq \lambda) \leq 
\exp\left(-\frac{\lambda^2}{2\mu(1-\mu)}\psi\left(\frac{\lambda}{\sqrt{n}\mu(1-\mu)}\right)\right)\label{ineq:bennett}
\end{align}
where $\psi (\lambda) := (2/\lambda^{2}) h(1+\lambda) $ where $h(\lambda) := \lambda(\log \lambda -1) +1$.
\end{bibn}


Inspecting the form of \eqref{LEON_PERRON_BINOMIAL_BOUND}, \eqref{TAL:BIN:BOUND}, and  \eqref{ineq:bennett},
we notice that 
when these bounds are compared to the hypergeometric tail bound \eqref{serfling:hypergeometric:spec} 
obtained from Serfling's bound they do not take advantage of the finite sampling setting.
Our hope then is we can derive probability bounds which look like the preceding binomial expressions, but improved by a
finite-sampling correction factor. Such improved bounds exist and are the main results of this paper. Their statements follow.

\begin{thm}\label{thm:lpanalogue}
Suppose $S_{n} \sim Hypergeometric(n,D,N)$. Define  $\mu := D/N$, and suppose $N > 4$ and $2 \leq n < D \leq N/2$.
Then for all $0< \lambda  < \sqrt{n}/2$ we have
\begin{align}
P\left(\sqrt{n}(\bar{X}_{n} - \mu) \geq \lambda\right) 
\leq
&\sqrt{\frac{1}{2\pi\lambda^2}} \left(\frac{1}{2}\right)
\sqrt{\left(\frac{N-n}{N}\right)\left(\frac{\sqrt{n}+2\lambda}{\sqrt{n}-2\lambda}\right)\left(\frac{N-n+2\sqrt{n}\lambda}{N-n-2\sqrt{n}\lambda}\right)}\nonumber \\
&\cdot\ \ \exp \left ( - \frac{2}{1-\frac{n}{N}} \lambda^2  \right ) \exp \left ( - \frac{1}{3} \left ( 1 + \frac{n^3}{(N-n)^3} \right ) \frac{\lambda^4}{n} \right )\ . \label{HYPERGEOMETRIC:LP_BOUND}
\end{align}
\end{thm}

\begin{thm}\label{thm:talagrandanalogue}
Suppose $\sum_{i=1}^nX_i \sim Hypergeometric(n,D,N)$. Define $\psi := n/N$ and $\mu := D/N$, and let $n < D$.
Fix $\mu_0,\psi_0 > 0$ such that $0 < \mu_0 \le \mu \le 1-\mu_0< 1$ and $0 < \psi_0 \le \psi  \le 1 - \psi_0 < 1$.
Then there exist constants $K_1,K_2$ depending only on $\mu_0$ and $\psi_0$ such that:
\begin{align}
&(i)\ \text{For all } \lambda> 0, \ \ \ 
P( \sqrt{n} ( \overline{X}_n - \mu) = \lambda )  
\leq \frac{K_1}{\sqrt{n}} \exp \left ( - \frac{2 \lambda^2}{1- \frac{n}{N} } \right )
       \exp \left ( - \left ( \frac{1}{4}  + \frac{1}{3} \left ( \frac{n}{N-n} \right )^3 \right ) \frac{\lambda^4}{n} \right )\ .  \nonumber \\
\nonumber \\
&(ii)\ \text{For all } 0 < t < \lambda, \ \ \ 
P( \sqrt{n} ( \overline{X}_n - \mu) \ge t )  
\leq  \frac{K_2}{\lambda}
\begin{pmatrix}
\exp \left ( - \frac{2 \lambda^2}{1- \frac{n}{N} } \right ) 
\exp \left ( -  \left ( \frac{1}{4} + \frac{1}{3} \left ( \frac{n}{N-n} \right )^3 \right ) \frac{\lambda^4}{n} \right ) \\
\ \  \cdot \exp \left ( \lambda ( \lambda -t) \left ( \frac{4}{1- \frac{n}{N} } + 1 + \frac{4 n^3}{3 (N-n)^3} \right ) \right ) 
\end{pmatrix}\ . \nonumber \\
\nonumber \\
&(iii)\ \text{For all } \lambda> 0, \ \ \ 
P( \sqrt{n} ( \overline{X}_n - \mu) \ge \lambda )  
\leq  \frac{K_2}{\lambda} \exp \left ( - \frac{2 \lambda^2}{1- \frac{n}{N} } \right )
       \exp \left ( -  \left (\frac{1}{4} + \frac{1}{3}\left ( \frac{n}{N-n} \right )^3 \right ) \frac{\lambda^4}{n} \right)\ . \label{TAL:HYPER:BOUND}
\end{align}
\end{thm}

We are also able to obtain an analogue of Bennett's inequality by using an important representation of the 
hypergeometric distribution as a sum of independent Bernoulli random variables with different means.
This representation results from a special case of results established by Vatutin and Mikha{\u\i}lov \citep{MR681461} 
(also see Ehm \citep{MR1093412}, Theorem A, and Pitman \citep{MR1429082}). 

\begin{ehmthm}
  If $1 \leq n \leq D \wedge (N-D)$, then 
  \begin{align}
    S_{n,D,N} =_d \sum_{i=1}^n Y_i \label{ehm:representation}
  \end{align}
  where $Y_i \sim \mbox{Bernoulli} (\pi_i)$ are independent.
\end{ehmthm}



We may use this representation along with Bennett's inequality to obtain a Bennett-type exponential bound for Hypergeometric random variables
(this bound was also discussed earlier in \cite{greenewellner2015finite}, though without proof). The proof of this claim is short, so we will provide it here.
\begin{thm}\label{thm:bennettanalogue}    
Suppose $S_{n,D,N}\sim\text{Hypergeometric}(n,D,N)$ with $1 \leq n \leq D \wedge (N-D)$. Define $\mu_N := D/N$,
$\sigma_N^2 := \mu_N (1-\mu_N)$, and $1-f_n := 1- (n-1)/(N-1)$ is the finite-sampling correction factor.
Then for all $\lambda>0$ 
\begin{align}
P( \sqrt{n} (\bar{X}_{n,D,N} - \mu_N ) > \lambda ) 
\leq \exp \left ( - \frac{\lambda^2}{2 \sigma_N^2 (1-f_n)} \psi \left ( \frac{\lambda}{\sqrt{n} \sigma_N^2 (1-f_n)} \right ) \right) \label{hyper:bennett}
\end{align}
where $\psi (\lambda) := (2/\lambda^{2}) h(1+\lambda)$ and $h(\lambda) := \lambda(\log \lambda -1) +1$.
\end{thm}

\begin{proof}[{\bf Proof}]
Under the hypotheses it follows from \eqref{ehm:representation} that
\begin{align}
P\left( \sqrt{n} (\bar{X}_{n,D,N} - \mu_N ) > \lambda \right) 
&=
P\left( n^{-1/2} \sum_{i=1}^n(Y_i -\mu_i)   > \lambda \right)  \nonumber \\
&\leq 
\exp\left(-\frac{\lambda^2}{2\left(\frac{\sum_1^n \pi_i (1-\pi_i)}{n}\right)}\psi\left(\frac{\lambda\cdot n^{-1/2}}{\left(
\frac{\sum_1^n \pi_i (1-\pi_i)}{n}\right)}\right)\right) \label{application:ben}
\\
&=
\exp\left(-\frac{\lambda^2}{2\left(\frac{n\mu_N(1-\mu_N)(1-f_n)}{n}\right)}\psi\left(
\frac{\lambda\cdot n^{-1/2}}{\left(\frac{n\mu_N(1-\mu_N)(1-f_n)}{n}\right)}\right)\right) 
\nonumber \\
&=
\exp \left ( - \frac{\lambda^2}{2 \sigma_N^2 (1-f_n)} \psi \left ( \frac{\lambda}{\sqrt{n} \sigma_N^2 (1-f_n)} \right ) \right) \ . \nonumber
\end{align}

Note that \eqref{application:ben} follows by applying Bennett's inequality (his general inequality, rather than the binomial specialization), 
which is applicable since each $Y_i$ is independent $Bernoulli(\mu_i)$ and hence $Y_i - \mu_i \leq 1$ a.s. for $1 \leq i \leq n$. This gives the bound.
\end{proof}
Since $\psi (v) \ge 1/(1+ v/3)$ for all $v\ge 0$ (Shorack and Wellner \citep{MR838963}, proposition 1, page 441),
 Theorem~\ref{thm:bennettanalogue} immediately yields following Bernstein type tail bound.
\begin{coro}\label{bernsteinanalogue}  
With the same assumptions and notation as in Theorem~\ref{thm:bennettanalogue},
\begin{align}
P( \sqrt{n} (\bar{X}_{n,D,N} - \mu_N ) > \lambda ) 
\leq \exp \left ( - \frac{\lambda^2/2}{ \sigma_N^2 (1-f_n) + \frac{\lambda}{3 \sqrt{n}}}  \right) . \label{hyper:bernstein}
\end{align}
\end{coro}

Detailed proofs of the bounds \eqref{HYPERGEOMETRIC:LP_BOUND} and \eqref{TAL:HYPER:BOUND}
are provided in section \ref{sec:proofs}. The proofs of these two bounds are complicated and do not proceed by the 
Cram{\'{e}}r - Chernoff method. 
The proof of \eqref{HYPERGEOMETRIC:LP_BOUND} adapts the argument of Le{\'{o}}n and Perron 
for the binomial distribution to the hypergeometric case.
In adapting the argument, we derive an analogue of a well known binomial tail probability bound going 
back to at least Feller \citep[pp.~150-151]{FELLER1}: see
Lemma \ref{hgbinanalogue} for details. The proof of \eqref{TAL:HYPER:BOUND} adapts Talagrand's 
argument to the hypergeometric setting.  
The tools developed in the course of the proofs are specialized to the analysis of binomial coefficients. 
As such, they may prove useful in understanding how to analyze the tail of distributions such as the multinomial and 
multivariate hypergeometric by providing guidance for parametrizations which could appear in those 
settings after the application of Stirling's formula.

Note that if $N \nearrow \infty$ with $n$ fixed, \eqref{HYPERGEOMETRIC:LP_BOUND} yields  
a slight improvement of \eqref{LEON_PERRON_BINOMIAL_BOUND}, the bound of Le{\'{o}}n and Perron,
since it contains a quartic term in the exponential.
Recovery of this sort is exactly the behavior we would expect in the limit, 
since \eqref{LEON_PERRON_BINOMIAL_BOUND} bounds binomial probabilities 
and as $N\nearrow \infty$ with $n/N \rightarrow 0$ the hypergeometric law 
converges to the binomial.  
A similar limiting argument shows we may recover \eqref{TAL:BIN:BOUND} from \eqref{TAL:HYPER:BOUND}
as well as \eqref{ineq:bennett} from \eqref{hyper:bennett}.

Also observe that the bounds \eqref{HYPERGEOMETRIC:LP_BOUND} and \eqref{TAL:HYPER:BOUND} contain terms involving $1-n/N$, which incorporates 
information about the proportion of the population sampled into the bound.
This sampling fraction is sharper than the improvement conjectured in  Serfling's bound: $1-n/N < 1-(n-1)/(N-1) < 1-(n-1)/N$. 
For $\lambda > (\sqrt{n} (N-n))/(2 (2 N-n))$, the expression outside the
exponential terms in \eqref{HYPERGEOMETRIC:LP_BOUND} exceeds the non-exponential expression in \eqref{LEON_PERRON_BINOMIAL_BOUND}. However,
for such $\lambda$ the increase  in magnitude is compensated for by the $1-n/N$ term appearing in the exponent. 

Figure \ref{fig:figure1} demonstrates the benefit of including a finite-sampling correction factor inside the exponential term: 
when enough of the the population is sampled, the difference between the binomial and hypergeometric bounds 
can differ by as much as $1/4$ for specific deviation values. 
Figure \ref{fig:figure2} compares the performance of the new hypergeometric bounds to each other and to the bounds of Serfling
\eqref{serfbnd} and Hush and Scovel \eqref{hushscovelbound}. It also 
provides some insight as to when \eqref{HYPERGEOMETRIC:LP_BOUND} out-performs \eqref{hyper:bennett}
and vice-versa. The finite-but-unspecified constants appearing in \eqref{TAL:HYPER:BOUND} prevent its inclusion in the figures. 
Additionally, the constants are not immediately comparable to 
those in \eqref{TAL:BIN:BOUND} because they depend on how one chooses to truncate the sampling fraction and population proportion. 
The bound \eqref{TAL:HYPER:BOUND} demonstrates that the factor $1-n/N$ in the exponential may apply for all $\lambda > 0$ 
as long as a suitable leading constant is selected.

Chatterjee \citep{MR2288072} used Stein's method to derive very general concentration bounds for statistics based on 
random permutations.  For example, here is a restatement of his Proposition 1.1: 
let $\{ a_{i,j} : 1 \le i,j \le N\}$ 
be a collection of numbers in $[0,1]$ and let 
$S  \equiv \sum_{i=1}^N a_{i, \pi(i)} $ where $\pi \sim \ $uniformly on all 
permutations of $\{1, \ldots , N\}$.  Then
$$
P( | S - E(S) | \ge t ) \le 2 \exp \left ( - \frac{t^2}{4 E(S) + 2t} \right ) \ \ \ \mbox{for all} \ \ t > 0 .
$$
The statistic $S$ was first studied by Hoeffding \citep{MR0044058}.
The special case which yields the setting of Serfling's inequality is $a_{i,j} := 1_{[i \le n ]} c_j$ 
for $1 \le i,j\le N$ where $1 \le n < N$.  Then $S = \sum_{i=1}^n c_{\pi(i)} \stackrel{d}{=} S_n$ where 
$S_n = \sum_{i=1}^n X_i$ is as defined in the first paragraph of section 1 above.  In this special case
$E(S) = n\overline{c}_N = n \mu$ and Chatterjee's (Bernstein type) bound becomes
\begin{align}
P( n^{-1/2}  (S_n - n\mu) \ge \lambda ) \leq \exp \left ( - \frac{\lambda^2}{ 4 \overline{c}_N + 2 \lambda / \sqrt{n}} \right ) 
\label{chatterjee:bound}
\end{align}
for all $\lambda > 0$.  

Goldstein and I\c{s}lak
\citep{MR3162712} recently used a variant of Stein's method 
to give another inequality for the tails of Hoeffding's statistic $S$:
\begin{eqnarray}
P( | S - E(S) | > t) \le 2 \exp \left ( - \frac{t^2}{2 (\sigma_A^2 + 8 \| a \| t )} \right ) 
\label{GoldsteinIslak}
\end{eqnarray}
where $\| a \| \equiv \max_{i,j \le N}  | a_{i,j} - a_{i \cdot } | $, 
\begin{eqnarray*}
a_{i \cdot}  & = & \frac{1}{N} \sum_{j=1}^N a_{ij} ,\ \ \ 
a_{\cdot j}  =   \frac{1}{N} \sum_{i=1}^N a_{ij} , \ \ \ 
a_{\cdot \cdot }  =  \frac{1}{N^2} \sum_{i,j=1}^N a_{ij} ,  \ \ \mbox{and} \\
\sigma_A^2 & = & \frac{1}{N-1} \sum_{i,j \le N} (a_{ij} - a_{i\cdot} - a_{\cdot j} + a_{\cdot \cdot} )^2 .
\end{eqnarray*}
Specializing (\ref{GoldsteinIslak}) to the setting of Serfling's inequality (with  $a_{i,j} := 1_{[i \le n ]} c_j$) yields
\begin{eqnarray}
P(n^{-1/2} | S_n - n \overline{c}_{N} | > \lambda) 
\le 2 \exp \left ( - \frac{\lambda^2/2}{ \sigma_c^2 (1-f_n)  + 8 \| c \| \lambda /\sqrt{n} } \right ) 
\label{GoldsteinIslakSpecCase1}
\end{eqnarray}
where $\sigma_c^2 = N^{-1} \sum_{j=1}^N (c_j - \overline{c}_N)^2$ 
and $\| c \| \equiv \max_{j\le N} | c_j - \overline{c}_N |$.  
This Bernstein type bound is in the same setting as Serfling's inequality, but the bound has an explicit dependence on 
$\sigma_c^2$.  This is similar to the bound of Bardenet and Maillard \eqref{expbnd:bm2} which incorporates variance
information through the parameter $\gamma^2$.  

Further specialization of (\ref{GoldsteinIslakSpecCase1}) to the (one-sided) hypergeometric setting 
(with $c_j = 1 \{ j \le D\}$ for $j=1, \ldots , N$) yields
\begin{eqnarray}
\lefteqn{P(n^{-1/2} ( S_n - n (D/N) ) > \lambda) } 
\label{GoldsteinIslakSpecCase2}\\
& \le & \exp \left ( - \frac{\lambda^2/2}{ (D/N)(1-D/N) (1-f_n)  + 8 \{(D/N) \vee (1-D/N)\} \lambda /\sqrt{n} } \right ) \nonumber\\
& = &  \exp \left ( - \frac{\lambda^2/2}{ \sigma_N^2 (1-f_n)  + 8 \{ \mu_N \vee (1- \mu_N)\} \lambda /\sqrt{n} } \right ) . \nonumber
\end{eqnarray}
This bound differs from the bound given in \eqref{hyper:bernstein} (the Bernstein type corollary of 
Theorem~\ref{thm:bennettanalogue}) only through the second term in the denominator 
inside the exponential:  note that $8 \{ \mu_N \vee (1- \mu_N)\} \ge 4 > 1/3$. 

Comparing the Bernstein type bounds (\ref{GoldsteinIslakSpecCase1}) and (\ref{GoldsteinIslakSpecCase2})
to Serfling's inequality \eqref{serfling:hypergeometric:spec} with $b=1$ and $a=0$, 
we see that the bound of Goldstein and I{\c{s}}lak is smaller than
Serfling's bound when $\lambda \leq \sqrt{n}/(32(\overline{c}_N \vee (1-\overline{c}_N)))(1-4\sigma_N^2+4\sigma_N^2f_n-f_n^*)$.
Similarly, we see that Chatterjee's bound \eqref{chatterjee:bound} 
is smaller than Serfling's bound only if $\overline{c}_N \le (1 - (n-1)/N)/8$ and then 
$\lambda \le \sqrt{n} ( 1 - (n-1)/N - 8 \overline{c}_N )/4$.
Figure \ref{fig:figure3} gives a comparison of 
Serfling's bound, Chatterjee's bound, Bardenet and Maillard's bound \eqref{expbnd:bm2}, Goldstein and I{\c{s}}lak's bound (\ref{GoldsteinIslakSpecCase2}),
and the Bennett type bound \eqref{hyper:bennett} in the further hypergeometric special case 
with $n=100$, $N=2001$, and $D \in \left\{101, 200\right\}$;  note that in the case $D=200$,
$\overline{c}_N = D/N \approx .10$ so $8\overline{c}_N \approx .8$ 
while $1-(n-1)/N) \approx 1 - .05$ so the first condition holds
and then Chatterjee's bound should win approximately when $\lambda \le \sqrt{n} (.15)/4 \approx 1.5/4$.  

Comparing (\ref{GoldsteinIslakSpecCase2})
to \eqref{chatterjee:bound},  we find the Goldstein-I{\c{s}}lak bound improves Chatterjee's bound when
 $\lambda \leq \sqrt{n}(2\overline{c}_N-\sigma_N^2(1-f_n))/(8(\overline{c}_N \vee (1-\overline{c}_N))-1)$. 
 In figure \ref{fig:figure3}, this region
is approximately equal to $\lambda \leq 0.08$  when $D=101$ and $\lambda \leq 0.18$ when $D=200$. 
From Figure~\ref{fig:figure3}(b) we see that the improvement of Goldstein and I{\c{s}}lak's bound to those of
Chatterjee and Serfling is very small in this region. 
From Figure~\ref{fig:figure3}(a) we see that Chatterjee's bound is smaller than 
both the Goldstein and I{\c{s}}lak bound (\ref{GoldsteinIslakSpecCase2}) as well as Serfling's bound, when $D/N$ is small and 
$0.08 \le \lambda \le \sqrt{n}(.55)/4 \approx 1.37$, 
but that all three are improved by Bardenet and Maillard's bound \eqref{expbnd:bm2}
and the Bennett type bound \eqref{hyper:bennett}.

\begin{figure}
  \begin{subfigure}{\textwidth}
    \centering
    \includegraphics[width=\linewidth,height=7.5cm,keepaspectratio]{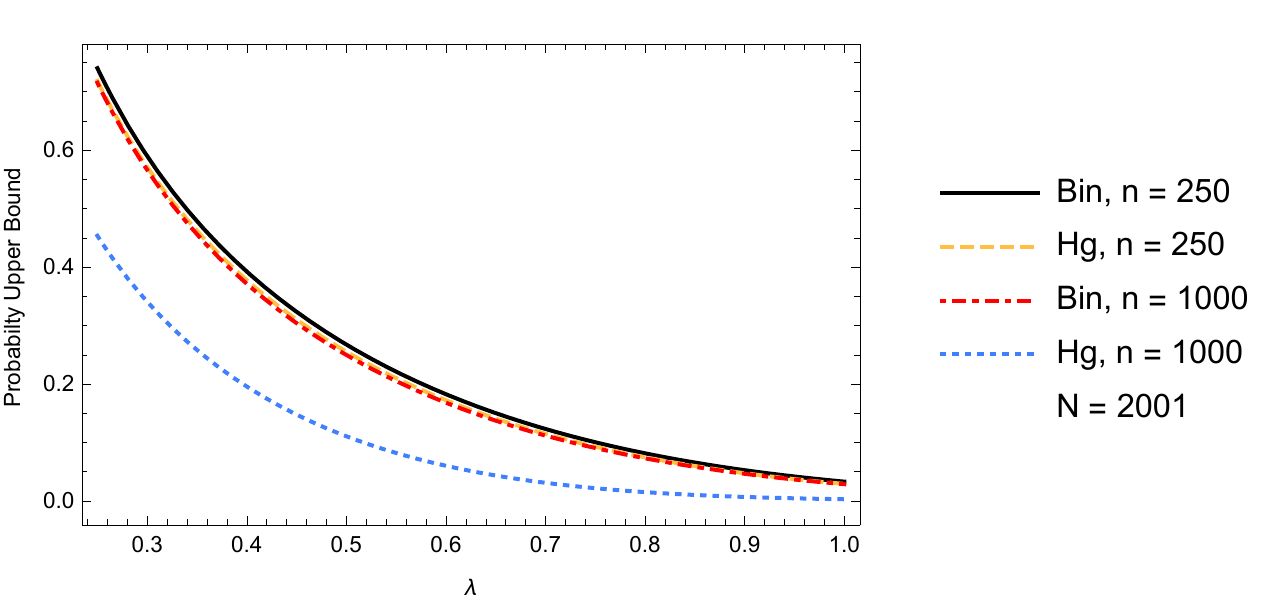}
    \caption{1a}
    \label{fig:sfig1}
  \end{subfigure}\\
  \begin{subfigure}{\textwidth}
    \centering
    \includegraphics[width=\linewidth,height=7.5cm,keepaspectratio]{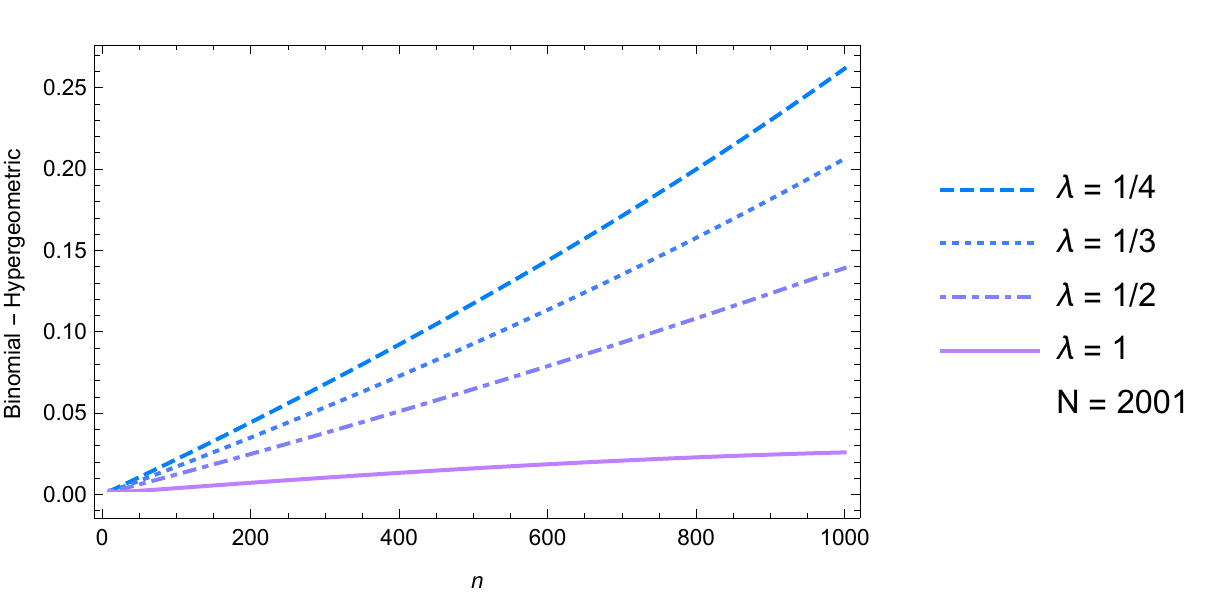}
    \caption{1b}
    \label{fig:sfig2}
  \end{subfigure}
\caption{Comparison of Le{\'{o}}n and Perron's binomial bound \eqref{LEON_PERRON_BINOMIAL_BOUND}  
to the new hypergeometric bound \eqref{HYPERGEOMETRIC:LP_BOUND}. 
In sub-figure \ref{fig:sfig1}, the sample size $n$ is set to $250$ and $1000$ for both
bounds. The population size $N$ is taken to be $2001$ in both cases.  In the legend, lines with the description ``Bin'' correspond to 
the binomial bound of Le{\'{o}}n and Perron \eqref{LEON_PERRON_BINOMIAL_BOUND}, while lines with the description ``Hg'' 
correspond to the new hypergeometric bound \eqref{HYPERGEOMETRIC:LP_BOUND}.
In sub-figure \ref{fig:sfig2}, we plot the difference between  Le{\'{o}}n and Perron's binomial bound 
\eqref{LEON_PERRON_BINOMIAL_BOUND}  
to the new hypergeometric bound \eqref{HYPERGEOMETRIC:LP_BOUND} 
at the fixed deviation-values $\lambda \in \left\{1/4,1/3,1/2,1\right\}$.
We let $n$ vary between $10$ and $1000$ to illustrate the impact of introducing the finite-sampling 
correction factor into the exponential term of the probability bound.
} 
\label{fig:figure1}
\end{figure}

\begin{figure}
  \begin{subfigure}{\textwidth}
    \centering
    \includegraphics[width=\linewidth,height=7.5cm,keepaspectratio]{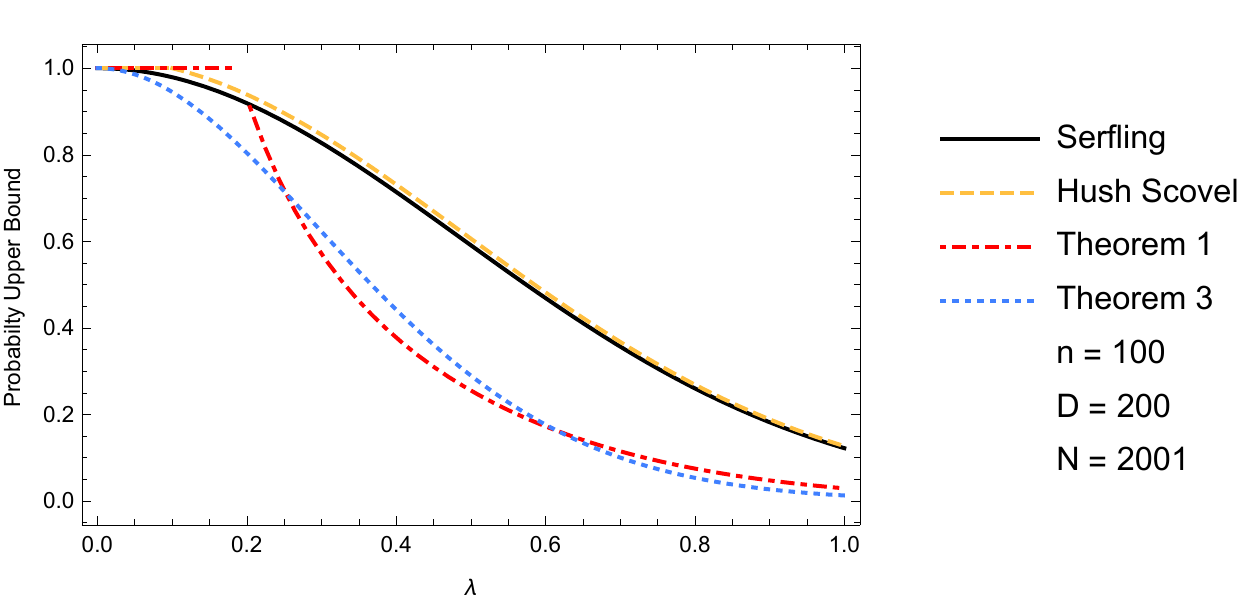}
    \caption{2a}
    \label{fig2:sfig1}
  \end{subfigure}
  \begin{subfigure}{\textwidth}
    \centering
    \includegraphics[width=\linewidth,height=7.5cm,keepaspectratio]{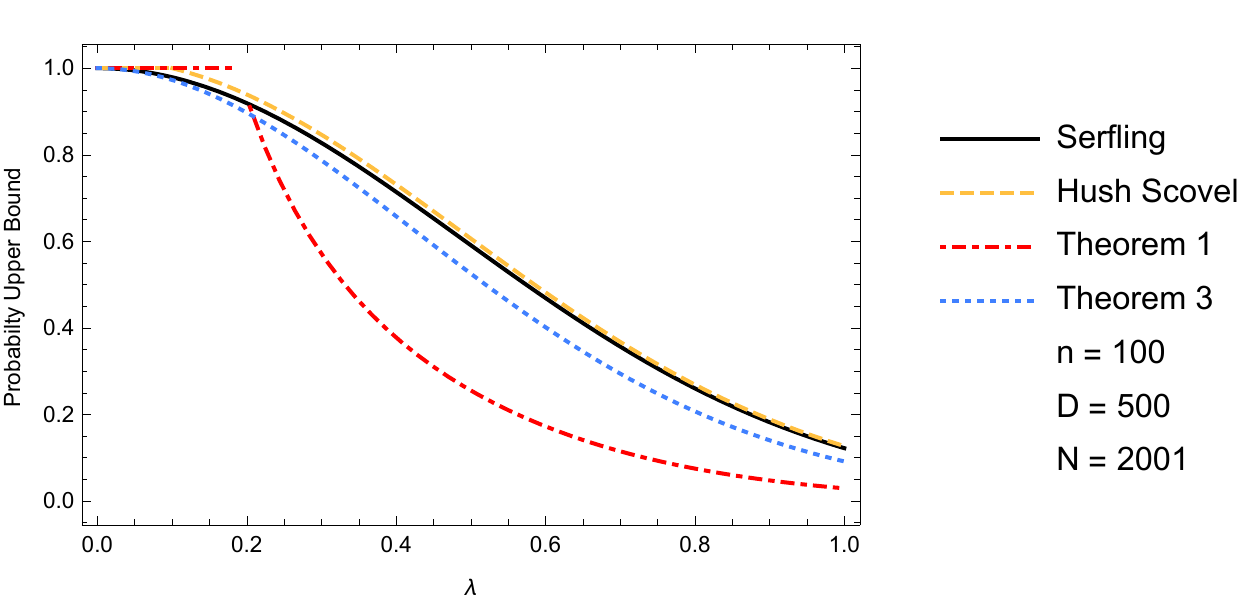}
    \caption{2b}
    \label{fig2:sfig2}
  \end{subfigure} \\
\caption{These plots compare the various exponential bounds for the Hypergeometric distribution. 
In these plots we fix the population to $N=2001$, and the sample size to $n=100$. 
The plots consider a setting with smaller variances by setting $D=200$ in the first plot (so $D/N=1/10$)
and $D=500$ in the second (so $D/N=1/4$). 
We see that the bound of Theorem \ref{thm:lpanalogue} \eqref{HYPERGEOMETRIC:LP_BOUND}  
performs comparably with the bound of Theorem \ref{thm:bennettanalogue}  \eqref{hyper:bennett} 
in the setting $D=200$ (2a), and surpasses it when $D=500$ (2b). This suggests that when $D/N < 1/10$, 
the bound of Theorem \ref{thm:bennettanalogue}  will perform better than the bound of Theorem \ref{thm:lpanalogue}, 
and when $1/10 \leq D/N < 1/2$, the converse. 
} 
\label{fig:figure2}
\end{figure}

\begin{figure}
  \begin{subfigure}{\textwidth}
    \centering
    \includegraphics[width=\linewidth,height=7.5cm,keepaspectratio]{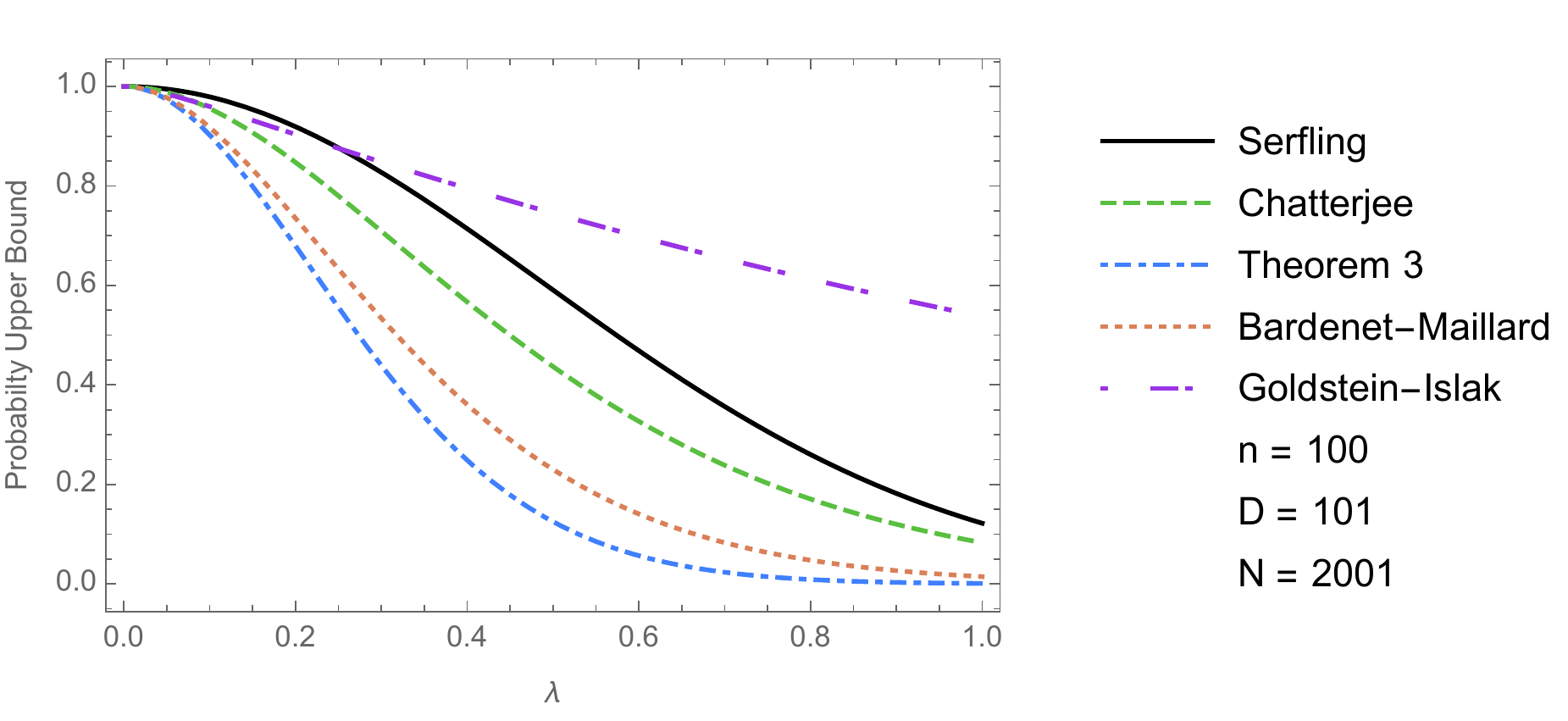}
    \caption{1a}
    \label{fig3:sfig1}
  \end{subfigure}\\
  \begin{subfigure}{\textwidth}
    \centering
    \includegraphics[width=\linewidth,height=7.5cm,keepaspectratio]{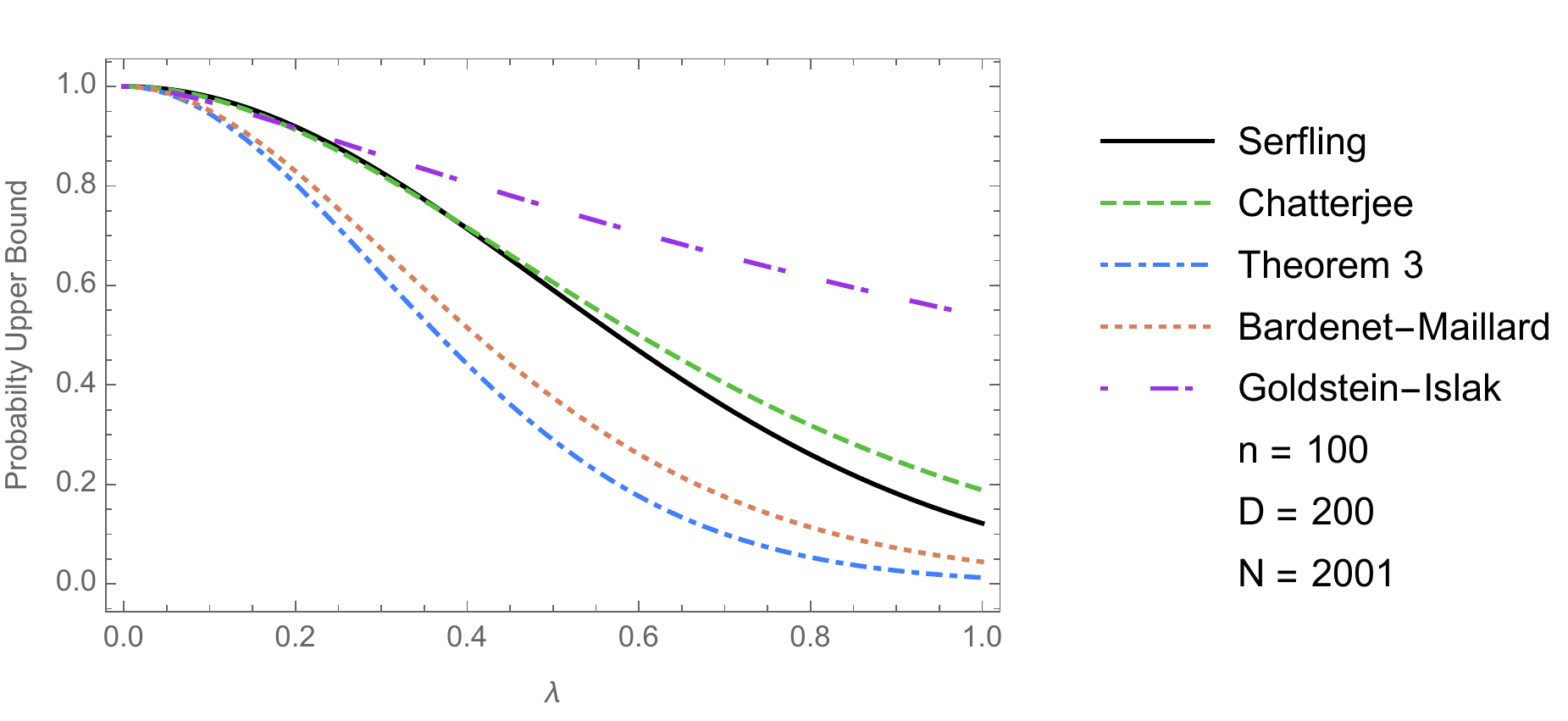}
    \caption{1b}
    \label{fig3:sfig2}
  \end{subfigure}
\caption{Comparison of Serfling's bound \eqref{serfling:hypergeometric:spec}, Chatterjee's bound \eqref{chatterjee:bound}, 
and the bound of Goldstein and I{\c{s}}lak \eqref{GoldsteinIslakSpecCase2},
Bardenet and Maillard's bound  \eqref{expbnd:bm2},
and Theorem \ref{thm:bennettanalogue}. In sub-figure \ref{fig3:sfig1}, the sample size is $n=100$, the population size is $N=2001$,
and the number of successes is $D=101$. In sub-figure \ref{fig3:sfig2}, the sample size remains $n=100$, 
the population size remains $N=2001$, but $D=200$.} 
\label{fig:figure3}
\end{figure}

\section{Convex Order for the Hypergeometric Distribution}
\label{sec:ConvexOrder}

When sampling without replacement from a finite population concentrated on $[0,1]$, the hypergeometric distribution occupies 
an extreme position with respect to convex order. 
This extreme position offers additional reason to give the hypergeometric distribution special consideration,
since we might hope to adapt bounds for its tail to the tails of the random variables it dominates through the convex order.

The extreme position of the hypergeometric distribution was essentially proved by Kemperman \citep{kemperman1973moment}. 
In his paper Kemperman studied (among many other things) finite populations majorized by nearly Rademacher populations; 
through transformation, this describes the hypergeometric setting.
We say nearly Rademacher since Kemperman's analysis resulted in majorizing populations 
consisting entirely of $-1's$ and $1's$ with the exception of a single exceptional element $\alpha$ with $-1 < \alpha < 1$. 

Here, we revisit his argument, modified so it applies to a population with elements between $0$ and $1$. 
We then provide an extension of the argument in order to obtain a hypergeometric population which sub-majorizes this initial population.
Since the extension follows naturally from Kemperman's majorization result, we begin with his procedure here. 
We start with relevant definitions from Marshall, Olkin, and Arnold \citep{marshall2010inequalities}.

\begin{defn}
For a vector $\bold{x} = (x_1,\dots,x_N) \in \mathbb{R}^N$, let
\[
x_{[1]} \geq x_{[2]} \geq \dots \geq x_{[N]}
\]
denote the components of $\bold{x}$ in decreasing order.
\end{defn}
\begin{defn}
For $\bold{x},\bold{y} \in \mathbb{R}^N$, 
\[
\bold{x} \prec \bold{y} \ \text{ if } \ 
\begin{cases}
\sum_{i=1}^k x_{[i]} \leq \sum_{i=1}^ky_{[i]}, \ k=1,\dots,N-1, \\
\sum_{i=1}^N x_{[i]} = \sum_{i=1}^Ny_{[i]} \\
\end{cases}
\]
where $\bold{x} \prec \bold{y}$ is read as ``$\bold{x}$ is majorized by $\bold{y}$''.
\end{defn}

\begin{defn}
For $\bold{x},\bold{y} \in \mathbb{R}^N$, 
\[
\bold{x} \prec_w \bold{y} \ \text{ if } \ 
\sum_{i=1}^k x_{[i]} \leq \sum_{i=1}^ky_{[i]}, \ k=1,\dots,N
\]
where $\bold{x} \prec_w \bold{y}$ is read as ``$\bold{x}$ is weakly sub-majorized by $\bold{y}$'' or, more briefly,
``$\bold{x}$ is sub-majorized by $\bold{y}$''.  
\end{defn}

Figure \ref{fig:figuremaj} provides an illustration of these definitions. 
In the following Lemma, we re-state Kemperman's procedure so it constructs a majorizing hypergeometric population.
See section 4, pages 165--168 in \citep{kemperman1973moment} for the original Rademacher argument.

\begin{figure}
    \centering
    \includegraphics[width=\linewidth,height=7.5cm,keepaspectratio]{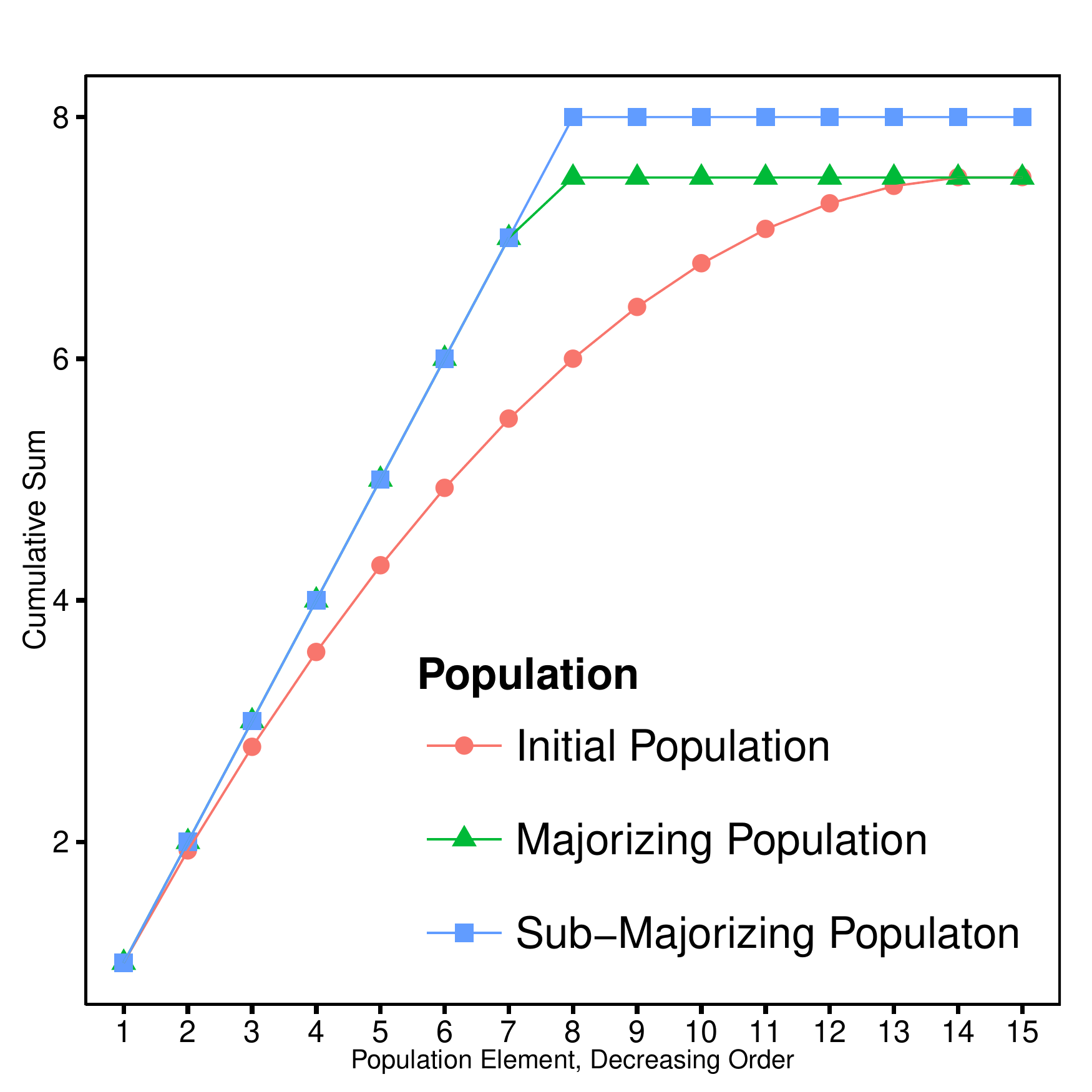}
    \caption{The initial population line in the display coresponds to $\bold{c}=\left\{0,1/14,2/14,\dots,13/14,1\right\}$. 
    The majorizing population contains seven $0's$, seven $1's$, and a single exceptional element of $1/2$. 
    The sub-majorizing population contains seven $0's$ and eight $1's$. In the display, each population is sorted in decreasing order; the corresponding
    lines show the cumulative sum of the ordered population elements.} 
    \label{fig:figuremaj}
\end{figure}

\begin{lem}\label{majlem1}
(\textbf{Kemperman} \citep{kemperman1973moment}) For any finite population 
$\bold{x}\in \mathbb{R}^N$, such that $0 \leq x_i \leq 1$ for all $1 \leq i \leq N$, there exists
a population $\bold{c} \in \mathbb{R}^N$, consisting only of $0$'s, $1$'s, 
and at most a single element between $0$ and $1$, which majorizes the original
population. In fact, $\bold{c}$ consists of $D$ $1$'s, $N-D-1$ $0$'s, and a number $\alpha \in [0,1)$ 
where $D$ and $\alpha$ are determined by 
$D= \lfloor N \bar{\bold{x}}_N\rfloor$, and $\alpha = N\bar{\bold{x}}_N - D$.
\end{lem}

\begin{algorithm}\label{alg1}
  \caption{Kemperman's majorization algorithm}\label{majpop}
  \begin{algorithmic}[1]
    \Function{Majorize}{$ipop$}\Comment{$ipop$ is the input population}
      \State $plen \gets \text{length}(ipop)$
      \State $mpop \gets ipop$ \Comment{Make a copy of the input population to transform}
      \For{$i \in \left\{1,\dots,plen-1\right\}$}
        \State $csum \gets mpop[i]+mpop[i+1]$
        \If{$csum > 1$}
          \State $mpop[i] \gets 1$
          \State $mpop[i+1] \gets (csum-1)$
        \Else
          \State $mpop[i] \gets 0$
          \State $mpop[i+1] \gets csum$
        \EndIf
      \EndFor
      \State \textbf{return} $mpop$\Comment{mpop is now transformed into the desired population}
    \EndFunction
  \end{algorithmic}
\end{algorithm}

\begin{proof}[{\bf Proof}]
We update Kemperman's argument, and demonstrate his modified
algorithm described in the display produces the population claimed by Lemma~\ref{majlem1}.
Suppose first that $\bold{x} \in \mathbb{R}^2$, and $0 \leq x_1,x_2 \leq 1$. Identify $ipop = \bold{x}$ in the algorithm description.
Then $mpop = \bold{x}$, $plen = 2$, and $plen-1=1$. Hence, the ``for'' loop executes exactly once. 

Consider the operations in  the ``for'' loop. We compute $csum = mpop[1]+mpop[2] = x_1+x_2$.
If $csum > 1$, the first condition is met, and we set $mpop[1] = 1$ and $mpop[2] = csum -1$. Since $csum = x_1+x_2$, and $0 \leq x_1,x_2 \leq 1$
by assumption, we have $1 < csum \leq 2$. Hence $0 < csum -1 \leq 1$, and so $0 < mpop[2] \leq 1$. 
Observe, that $mpop[1] = 1 \geq x_1\vee x_2,$ $mpop[1]+mpop[2] = 1 + (x_1+x_2 -1) = x_1+x_2$, so that mpop now majorizes $\bold{x}$ as claimed.
If $csum \leq 1$, the algorithm sets $mpop[2] = x_1+x_2$, and $mpop[1] = 0$. Again, this satisfies the description in the lemma, since $0 \leq csum =x_1+x_2 \leq 1$.
Moreover, $mpop[2] = x_1+x_2 \geq x_1\vee x_2$ and $mpop[2]+mpop[1] = x_1+x_2+0 = x_1+x_2$, and so again $mpop$ majorizes the population $\bold{x}$.
This completes the base case. Observe that the exceptional element is in the final index of the vector.


For the inductive case, suppose Kemperman's algorithm works when $n=N$. We will show it holds for $n=N+1$.
Let $\bold{x} \in \mathbb{R}^{N+1}$ be a population whose elements are all between $0$ and $1$. 
Let $\bold{y} \in \mathbb{R}^N$ be constructed so that $y_i = x_i$ for $1\leq i \leq N$. Run the algorithm on $\bold{y}$. By the induction hypothesis,
this produces a vector $\bold{m} \in \mathbb{R}^N$, such that $m_i \in \left\{0,1\right\}$ for $1\leq i \leq N-1$. Moreover, $0 \leq m_N \leq 1$ by the induction hypothesis, and also $\bold{m}$
majorizes $\bold{y}$. 

Next, construct a vector $\bold{a} \in \mathbb{R}^2$ such that $a_1 = m_N$ and $a_2 = x_{N+1}$. Run the algorithm on $\bold{a}$. By the base case, this produces a new vector $\bold{b} \in \mathbb{R}^2$
such that $b_1 \in \left\{0,1\right\}$ and $0 \leq b_2 \leq 1$. Note also that $\bold{b}$ majorizes $\bold{a}$. 

Finally, construct a vector $\bold{c} \in \mathbb{R}^{N+1}$ such that $c_i = m_i$ for $1 \leq i \leq N-1$, $c_N=b_1$ and $c_{N+1}=b_2$. By construction, we have that $c_i \in \left\{0,1\right\}$ for 
$1 \leq i \leq N$ and $0 \leq c_{N+1} \leq 1$. Hence, if we can show that $\bold{c}$ majorizes $\bold{x}$ we are done. 
We first show that 
\[
\sum_{i=1}^{N+1} c_i = \sum_{i=1}^{N+1} x_i\ .
\]
Using the constructions, we have
\begin{align}
\sum_{i=1}^{N+1} c_i &= \left[\sum_{i=1}^{N-1} c_i\right] + c_{N} + c_{N+1} = \left[\sum_{i=1}^{N-1} m_i\right] + b_{1} + b_{2} 
= \left[\sum_{i=1}^{N-1} m_i\right] + m_{N} + x_{N+1}  \nonumber \\
&= \left[\sum_{i=1}^{N} m_i\right]  + x_{N+1} 
= \left[\sum_{i=1}^{N} y_i\right]  + x_{N+1}  
= \left[\sum_{i=1}^{N} x_i\right]  + x_{N+1} 
= \sum_{i=1}^{N+1} x_i\ , \nonumber
\end{align}
and so the summation claim holds. Next, pick $1 \leq k \leq N+1$. Then
\[
\sum_{i=1}^{k} c_{[i]} \geq \sum_{i=1}^{k} x_{[i]}
\]
since by construction, $c_i \in \left\{0,1\right\}$ for $1 \leq i \leq N$. As $0 \leq x_i \leq 1$ for $1 \leq i \leq N+1$, if $k$ is small enough so that
\[
\sum_{i=1}^{k} c_{[i]} < \sum_{i=1}^{N+1} c_{[i]}=\sum_{i=1}^{N+1} c_{i}\ ,
\]
then all terms in the summation
\[
\sum_{i=1}^{k} c_{[i]}
\]
must be $1$, and so are greater than or equal to the corresponding $x$'s. If $k$ is large enough so that
\[
\sum_{i=1}^{k} c_{[i]}=\sum_{i=1}^{N+1} c_{i}
\]
(which means the remaining $N+1-k$ elements are all $0$), then we also have
\[
\sum_{i=1}^{k} c_{[i]} \geq \sum_{i=1}^{k} x_{[i]}
\]
since we have already seen that the sum over all the $x$'s equals the sum over all the $c$'s.
As all the $x$'s are between $0$ and $1$ by assumption, this proves the claim.
\end{proof}

%
%
%
%

\begin{lem}{\label{submaj}}
For any finite population $\bold{x}\in \mathbb{R}^N$, such that $0 \leq x_i \leq 1$ for all $1 \leq i \leq N$, there exists
a population $\bold{z} \in \mathbb{R}^N$, consisting only of $0$'s and $1$'s, which sub-majorizes the original population.
\end{lem}
\begin{proof}[{\bf Proof}]
Consider a finite population $\bold{x} \in \mathbb{R}^N$ which obeys the hypotheses.
Using Lemma \ref{majlem1}, we may construct a population $\bold{y}\in \mathbb{R}^N$ which majorizes $\bold{x}$. 
By Lemma \ref{majlem1}, we know that $\bold{y}$ consists only of $0$'s, $1$'s, and at most a 
single exceptional element $y_N$ between $0$ and $1$. 

If the exceptional element is either exactly $0$ or exactly $1$, we are done. 
So, suppose $0 < y_N < 1$. Create a new population $\bold{z} \in \mathbb{R}^N$
such that $z_i = y_i$ for $1 \leq i \leq N-1$, and $z_N = 1$. 
This population $\bold{z}$ then sub-majorizes $\bold{y}$ and hence sub-majorizes $\bold{x}$,
completing the proof.
\end{proof}

\begin{lem}
Suppose $\bold{x} \in \mathbb{R}^N$ is a population consisting only of $0$'s, $1$'s, and a single exceptional element, $x_1$, such that $0 < x_1 < 1$.
Suppose $\bold{y} \in \mathbb{R}^N$ is a population whose elements are the same as those in $\bold{x}$, except $y_1 = 1$ and so $y_1 > x_1$. 
Let $X_1,\dots,X_n$ denote a sample without replacement from $\bold{x}$, and $Y_1, \dots, Y_n$ denote a sample without replacement from $\bold{y}$, $1 \leq n \leq N$.
Finally, suppose $\phi$ is a continuous convex increasing function on $\mathbb{R}$. Then
\begin{align}
E\phi\left(\sum_{i=1}^nX_i\right) \leq E\phi\left(\sum_{i=1}^nY_i\right)\ . \label{submajin}
\end{align}
\end{lem}
\begin{proof}[{\bf Proof}]
We adapt Kemperman's (1973) argument for Rademacher populations to the current setting of hypergeometric sub-majorization. 

Observe that
\[
E\phi\left(\sum_{i=1}^nX_i\right) = \frac{1}{{N\choose n}} \sum \phi\left(x_{i_1}+\dots+x_{i_n}\right)
\]
where the sum is over all sets of indices $1 \leq i_1 < i_2 < \dots < i_n \leq N$. Note the same holds for sampling without replacement from $\bold{y}$, with
suitable substitution. 
Therefore
\begin{align}
{{N\choose n}}\left[E\phi\left(\sum_{i=1}^nY_i\right) - E\phi\left(\sum_{i=1}^nX_i\right)\right] 
&=
\sum 
\begin{pmatrix}
\phi\left(y_1 + y_{i_2} + \dots +x_{i_n}\right) 
-
\phi\left(x_1 + x_{i_2} + \dots +x_{i_n}\right)\\ 
\end{pmatrix}
\nonumber 
\end{align}
where the sum is over all distinct indices $2 \leq i_2 < i_3 < \dots < i_n \leq N$. Note that sets of indices with $i_1>1$ cancel out by definition
of the two populations. Since $\phi$ is assumed convex increasing, each term of the sum is non-negative. Hence, the entire sum is non-negative as well. This 
gives the claim.
\end{proof}
We next specialize a proposition stated in Marshall, Olkin, and Arnold \citep[p.~455]{marshall2010inequalities} to the current problem. 
Proof of the general statement given in the text is credited to Karlin;
proof for the specific cases of sampling with and without replacement to Kemperman. 
As proof is given in Marshall, Olkin, and Arnold, we simply state the result here.

\begin{lem}\label{moalem}
Let $\bold{x}\in \mathbb{R}^N$ be an arbitrary finite population.
Let $\bold{y} \in \mathbb{R}^N$ be a finite population which majorizes $\bold{x}$. 
Let $X_1,\dots,X_n$ denote a sample without replacement from $\bold{x}$, and $Y_1, \dots, Y_n$ denote a sample without replacement from $\bold{y}$, $1 \leq n \leq N$.
Finally, suppose $\phi$ is a continuous convex increasing function on $\mathbb{R}$. Then
\[
E\phi\left(\sum_{i=1}^nX_i\right) \leq E\phi\left(\sum_{i=1}^nY_i\right)\ .
\]
\end{lem}

Note that Lemma \ref{moalem} requires majorization between populations. 
We may combine the preceding lemmas to demonstrate the following claim.

\begin{thm}\label{submajthm}
For any finite population $\bold{x}\in \mathbb{R}^N$, such that $0 \leq x_i \leq 1$ for all $1 \leq i \leq N$, there exists
a population $\bold{y} \in \mathbb{R}^N$, consisting only of $0$'s and $1$'s which sub-majorizes the original population.
Let $X_1,\dots,X_n$ denote a sample without replacement from $\bold{x}$, and $Y_1, \dots, Y_n$ denote a sample without replacement from $\bold{y}$, $1 \leq n \leq N$.
Finally, suppose $\phi$ is a continuous convex increasing function on $\mathbb{R}$. Then
\[
E\phi\left(\sum_{i=1}^nX_i\right) \leq E\phi\left(\sum_{i=1}^nY_i\right)\ .
\]
\end{thm}

\begin{proof}[{\bf Proof}]
Suppose $\bold{x} \in \mathbb{R}^N$ is a finite population which satisfies the hypotheses.  
We may use Lemma \ref{majlem1} to construct a population $\bold{z} \in \mathbb{R}^N$ such that
$\bold{z}$ majorizes $\bold{x}$, and $\bold{z}$ consists only of $0$'s, $1$'s, and at most a single exceptional element between $0$ and $1$.
For $1 \leq n \leq N$, let $Z_1,\dots,Z_n$ denote a sample without replacement from $\bold{z}$. By Lemma \ref{moalem}, we then have the order
\begin{align}
E\phi\left(\sum_{i=1}^nX_i\right) \leq E\phi\left(\sum_{i=1}^nZ_i\right)\ . \label{first_conv_order}
\end{align}
Next, by Lemma \ref{submaj} we may construct a population $\bold{y} \in \mathbb{R}^N$ consisting only of $0$'s and $1$'s that sub-majorizes $\bold{z}$.
Then by \eqref{submajin} we have
\begin{align}
E\phi\left(\sum_{i=1}^nZ_i\right) \leq E\phi\left(\sum_{i=1}^nY_i\right)\ . \label{second_conv_order}
\end{align}
Combining \eqref{first_conv_order} and \eqref{second_conv_order} proves the claim.
\end{proof}

Inequality \eqref{hyper:bennett} provides an opportunity to apply Theorem \ref{submajthm}. 
Recalling the notation of the introduction, 
let $\bold{c} := \left\{c_1,\dots,c_N\right\}$ be a population such that $0 \leq c_i \leq 1$ for $1 \leq i \leq N$, $a=0$, $b=1$, and $\overline{c}_N + 1/N \leq 1/2$.
Using Kemperman's algorithm, as stated in Lemma \ref{majlem1}, there exists a population
$\bold{m} := \left\{m_1,\dots,m_N\right\}$ which majorizes $\bold{c}$ such that $m_i \in \left\{0,1\right\}$ for $1 \leq i \leq N-1$, and $0 \leq m_N \leq 1$. In the following,
suppose $0 < m_N < 1$, since if $m_N=0$ or $m_N=1$ we can apply \eqref{hyper:bennett} directly.

Since $\bold{c} \prec \bold{m}$, we have $\overline{m}_N = \overline{c}_N$. Using Lemma 2, there is a population $\left\{h_1,\dots,h_N\right\}$ with
$h_i \in \left\{0,1\right\}$ for $1 \leq i \leq N$ that sub-majorizes $\bold{c}$. By the preceding construction, we have $h_i = m_i$ for
$1 \leq i \leq N-1$, and $h_N \equiv 1 > m_N$. 

Without loss of generality, re-label $\bold{m}$ and $\bold{h}$ so that: for $1 \leq i \leq D-1$ we have $h_i=m_i=1$; for $i=D$ 
we have $h_D=1 > m_D > 0$;
for $D+1 \leq i \leq N$ we have $h_i=m_i=0$. Denote the exceptional element of $\bold{m}$ by $\alpha := m_D$. 
With the populations so modified, we derive bounds for the difference between the populations means:
\begin{align}
\frac{1}{N} \geq\overline{h}_N - \overline{m}_N = \frac{h_D-m_D}{N} = \frac{1-\alpha}{N} \geq 0\ . \label{extmeandiff}
\end{align}
By construction, we thus have
\[
\overline{m}_N \leq \overline{h}_N \leq \overline{m}_N+\frac{1}{N} \leq \frac{1}{2}\ .
\]
Suppose then that we sample $n < D$ items without replacement from $\bold{c}$. Let $X_i$ denote the sample without
replacement from $\bold{c}$, and let $H_i$ denote a corresponding sample without replacement from $\bold{h}$. Then for $t > 0$ 
\begin{align}
P\left( \sum_{i=1}^nX_i - n \mu_c \ge t \right) 
 \le\ & \inf_{r>0}\frac{ E \exp \left( r \sum_{i=1}^nX_i \right)}{\exp \left(r t + r n \mu_c\right)} \nonumber \\
 \le\ & \inf_{r>0}\frac{ E \exp \left( r \sum_{i=1}^nH_i \right)}{\exp \left(r t + r n \mu_c\right)} \label{thm4use} \\
 =\ & \inf_{r>0}\exp\left(r n (\mu_h-\mu_c)\right)\frac{ E \exp \left( r   \sum_{i=1}^n(H_i-\mu_h)\right)}{\exp \left(r t\right)}  \nonumber \\
 \le\ & \inf_{r>0}\exp\left(r \frac{n}{N} \right)\frac{ E \exp \left( r \sum_{i=1}^n(H_i-\mu_h)\right)}{\exp \left(r t\right)}  \label{meanpush} \\
 = \ & \inf_{r>0}\exp\left(r \frac{n}{N} \right)\frac{ E \exp \left( r \sum_{i=1}^n(Y_i-\pi_i)\right)}{\exp \left(r t\right)}  \nonumber  \\
\le \ & \inf_{r>0}\exp\left(r \frac{n}{N}-rt+n\left(\frac{\sum_{i=1}^n\pi_i(1-\pi_i)}{n}\right)\left(e^r-1-r\right)\right) \label{swbnd} \\
=\ &\inf_{r>0}\exp\left(r \frac{n}{N}-rt+n \gamma^2\left(e^r-1-r\right)\right) \label{intbnd}
\end{align}
where in the final line we write $\gamma^2 := (1/n)\sum_{i=1}^n\pi_i(1-\pi_i)$. The inequality at \eqref{meanpush} follows by \eqref{extmeandiff}.
The inequality at \eqref{thm4use} follows by Theorem \ref{submajthm}.
The inequality at line \eqref{swbnd} follows by Shorack and Wellner page 852, display (b) \citep{MR838963}.

At this point we may continue from \eqref{intbnd} and  optimize over $r$. Doing so yields an optimal choice of
\[
r^* = \log\left(1+\frac{Nt-n}{nN\gamma^2}\right)\ .
\]
Using this value, however, yields an exponential bound that is somewhat difficult to compare to \eqref{hyper:bennett}. If instead we simply choose
\[
r^*_2 = \log\left(1+\frac{t}{n\gamma^2}\right)\ ,
\]
we obtain a bound similar in performance to the bound we find using $r^*$, but has the benefit of easy comparison to \eqref{hyper:bennett}. 
The choice $r_2^*$ corresponds to the optimal value of $r$ when the original population is majorized by a hypergeometric population. 
We continue from \eqref{intbnd} using $r_2^*$, and obtain
\begin{align}
P\left( \sum_{i=1}^nX_i - n \mu_c \ge t \right)  
\le & \exp\left(\frac{n}{N}\log\left(1+\frac{t}{n\gamma^2}\right)\right) 
\cdot  \exp\left(-t\left[\left(1+\frac{n\gamma^2}{t}\right)\log\left(1+\frac{t}{n\gamma^2}\right)-1\right]\right) \nonumber \\
= & \exp\left(\frac{n}{N}\log\left(1+\frac{t}{n\gamma^2}\right)\right)  \cdot \exp\left(-\frac{t^2}{2n\gamma^2}\psi\left(\frac{t}{n\gamma^2}\right)\right)\ .
\end{align}
Writing $\lambda = t/\sqrt{n}$\ , and substituting $\gamma^2= (D/N)(1-D/N)(1-f_n) \equiv \sigma_N^2(1-f_n)$, we obtain the following bound:
\begin{align}
P(\sqrt{n}(\overline{X}_n-\mu_c) \ge \lambda) \leq
\left(1+\frac{\lambda}{\sqrt{n}\sigma_N(1-f_n)}\right)^{n/N}  \exp\left(-\frac{\lambda^2}{2\sigma_N^2(1-f_n)}\psi\left(\frac{\lambda}{\sqrt{n}\sigma_N^2(1-f_n)}\right)\right)\ .\nonumber
\end{align}
In this form, the cost of sub-majorization is clear when compared to \eqref{hyper:bennett}: we incur the leading term outside the exponent.
By shifting and scaling the population, we may use this bound to 
obtain the following theorem for the general problem of sampling without replacement:
\begin{thm}\label{thm:genben}
Let $\bold{c} := \left\{c_1,\dots,c_N\right\}$ be a population  with $a=\min_{1 \leq i \leq N} c_i$ and $b=\max_{1\leq i \leq N}c_i$ both finite.
Let $\bold{d} := \left\{(c_1-a)/(b-a),\dots,(c_N-a)/(b-a)\right\}$. Suppose first that 
$\bold{d}$ is majorized by a Hypergeometric population such that $D/N \leq 1/2$. From this Hypergeometric population define $\sigma_N^2:=(D/N)(1-D/N)$.
Then the following bound holds for a sample without replacement of $n < D$ items from $\bold{c}$:
\begin{align}
P(\sqrt{n}(\overline{X}_n-\mu_c) \ge \lambda) \leq
\exp\left(-\frac{\lambda^2}{2(b-a)^2\sigma_N^2(1-f_n)}\psi\left(\frac{\lambda}{\sqrt{n}(b-a)\sigma_N^2(1-f_n)}\right)\right)\ . \label{exactmaj}
\end{align}
If instead $\overline{d}_N + 1/N \leq 1/2$, then the following bound holds for a sample without replacement of $n < D$ items from $\bold{c}$:
\begin{align}
\ \ &P(\sqrt{n}(\overline{X}_n-\mu_c) \ge \lambda)  \nonumber \\
\leq\ \ 
&\left(1+\frac{\lambda}{\sqrt{n}(b-a)\sigma_N(1-f_n)}\right)^{n/N}  \cdot \exp\left(-\frac{\lambda^2}{2(b-a)^2\sigma_N^2(1-f_n)}\psi\left(\frac{\lambda}{\sqrt{n}(b-a)\sigma_N^2(1-f_n)}\right)\right)\ .
\label{suboptsubmaj}
\end{align}
\end{thm}
Two-sample rank tests provide an opportunity to explore the behavior of the bounds of Theorem \ref{thm:genben}. 
Following the exposition in Chapter 4 of H{\'{a}}jek, {\v{S}}id{\'{a}}k, and Sen \citep{MR1680991} (with the notation modified), 
let $Y_1,\dots,Y_n$ and $Z_1,\dots,Z_m$ be random 
samples with continuous distributions $F_Y$ and $F_Z$. Form the pooled sample $Y_{n+j}=Z_j$, $j=1,\dots, m$, and $N=n+m$. 
Let $R_i\ (i=1,\dots,N)$ denote the rank of the observation $Y_i$ in the ordered sequence $Y_{(1)} < Y_{(2)} < \dots < Y_{(N)}$.  To test the null hypothesis
$H_0: F_Y=F_Z$ against alternatives of shifts in location, 
one may use the Wilcoxon test (see page 96 \citep{MR1680991}). 
The test statistic, expectation under the null, and variance under the null are
\[
S_W := \sum_{i=1}^n R_i\ , ES_W = \frac{1}{2}n(n+m+1)\ , \text{ and } Var(S_W)=\frac{1}{12}nm(n+m+1)\ .
\]
Under the null, $S_W$ may be viewed as the sum in a sample without replacement from the population $\bold{c}_W := \left\{1,2,\dots,N\right\}$, where $a = 1$ and $b=N$.
Shifting and scaling the population produces $\bold{d}_W := \left\{0,1/(N-1),\dots,(N-2)/(N-1),1\right\}$\ . 
If $N$ is even, then $\bold{d}_W$ is majorized by a Hypergeometric population
containing $N/2\ 1's$ and $N/2\ 0's$, and hence $\sigma^2_N = (D/N)(1-D/N)=1/4$. If we additionally assume $n \leq m$,
we may use \eqref{exactmaj} to study its finite sample behavior. Doing so we find for $\lambda > 0$
\[
P\left(\sqrt{n}\left(\overline{X}_n-\frac{(N+1)}{2}\right) \ge \lambda\right) \leq
\exp\left(-\frac{2\lambda^2}{(n+m-1)^2(1-f_n)}\psi\left(\frac{2\lambda}{\sqrt{n}(n+m-1)(1-f_n)}\right)\right)\ . 
\]
Serfling's bound \eqref{serfbnd} may be applied in this case as well; through its application we find
\[
P\left(\sqrt{n}\left(\overline{X}_n-\frac{(N+1)}{2}\right) \ge \lambda\right) \leq
\exp\left(-\frac{2\lambda^2}{(n+m-1)^2(1-\frac{n-1}{n+m})}\right)\ . 
\]
Finally, we may use Bardenet and Maillard's bound \eqref{expbnd:bm2} with $\delta= \delta_f=1\times 10^{-7}$ to analyze the situation as well.
Figure \ref{fig:figure5} compares the performance of these three bounds when $n=m=250$. In this case, we see that the bounds are comparable, 
with Bardenet and Maillard's bound performing the best, and Serfling's performance superior to \eqref{exactmaj}.
This occurs because the variance component $(D/N)(1-D/N)$ (which is close to $1/4$ when $n=m=250$) 
that appears in the bound is the variance of the majorizing hypergeometric 
population rather than the variance of the shifted and scaled population $\bold{d}_w$ (which is close to $1/12$ when $n=m=250$). 
Bardenet and Maillard's bound performs well because it incorporates information about the variance of the 
untransformed population into its bound.

Another example is found in the Klotz test, which is used to test the null $H_0: F_Y=F_Z$ against alternatives of differences in scale 
(see page 104 \citep{MR1680991}). Recalling $N := n+m$, 
the test statistic, expectation under the null, and variance under the null are
\[
S_K := \sum_{i=1}^n \left[\Phi^{-1}\left(\frac{R_i}{N+1}\right)\right]^2\ ,\ ES_K = \frac{n}{N}\sum_{i=1}^{N} \left[\Phi^{-1}\left(\frac{i}{N+1}\right)\right]^2\ , 
\]
and
\[
VarS_K = \frac{nm}{N(N-1)}\sum_{i=1}^{N} \left[\Phi^{-1}\left(\frac{i}{N+1}\right)\right]^4 - \frac{m}{n(N-1)}(ES_K)^2\ .
\]
Defining the population 
\[
\bold{c}_K:= \left\{c_i := \left[\Phi^{-1}\left(\frac{i}{N+1}\right)\right]^2\ ,\ 1 \leq i \leq N\right\}\ ,
\]
we may view $S_K$ under the null as the sum in a sample without replacement from $\bold{c}_K$. If $n+m=500$, 
we may compute $\bold{c}_K$, and find $a \approx 6.26 \times 10^{-6}$ and
$b \approx 8.29$. Shifting and scaling the population produces $\bold{d}_K$, 
which is bounded by $0$ and $1$. This population is majorized by a population containing $59$ $1's$,
$440$ $0's$, and a single exceptional element approximately equal to $0.044$. 
Hence, it is sub-majorized by a population containing $60$ $1's$ and $440$ $0's$. 
Supposing that $n=60$ and $m=440$, we may use \eqref{suboptsubmaj} to analyze this scenario 
since the mean of the sub-majorizing population is $3/25$ (also note $(D/N)(1-D/N)=66/625$ in this case).  
Doing so (with the conservative approximation that $b-a \approx 8.29$), 
we find for $\lambda \geq 0$ that
\begin{align}
\ \ &P(\sqrt{60}(\overline{X}_{60}-\mu_K) \ge \lambda)  \nonumber \\
\leq\ \ 
&\left(1+\frac{\lambda}{\sqrt{60}(8.29) (66/625) (440/499)}\right)^{3/25}  
\cdot \exp\left(-\frac{\lambda^2}{2(8.29)^2(66/625)(440/499)}\psi\left(\frac{\lambda}{\sqrt{60}(8.29)(66/625)(440/499)}\right)\right)\ . \nonumber
\end{align}
Once again, we may apply Serfling's uniform bound. Doing so here, we find
\begin{align}
P\left(\sqrt{60}(\bar{X}_{60} - \mu_K) \geq \lambda \right) \leq
\exp\left(-\frac{2 \lambda^2}{(441/500)(8.29)^2}\right)\ . \nonumber \\
\end{align}
As in the Wilcoxon example, 
we may use Bardenet and Maillard's bound \eqref{expbnd:bm2} with $\delta= \delta_f=1\times 10^{-7}$ to analyze the situation.
Figure \ref{fig:figure5} also compares the performance of these three bounds for the special case $n=60$ and $m=440$. 
In this case, we again see that Bardenet and Maillard's bound performs the best, 
but that the bound obtained via sub-majorization now improves on Serfling's result. This is because the variance component of the sub-majorizing
hypergeometric population, $(D/N)(1-D/N)=66/625 = 0.1056 < 1/4$ reflects some of the variability of the untransformed population. However,
the untransformed population $d_K$ has variance $\sigma^2 \approx 0.0258$; this variability is captured in the bound of Bardenet and Maillard, and
so we see the improved performance.

Thus we see that (sub-)majorization, as a strategy for finding exponential bounds which incorporate information about the population variance
in the problem of sampling without replacement from a bounded finite population, 
can produce sub-optimal results. As we saw, this is because the (sub-)majorizing hypergeometric population can be
more variable than the underlying population from which we sample. 
However, if our goal is to find uniform exponential bounds, this information loss is immaterial: such
bounds apply to all underlying populations, regardless of their variability. 
Hence the analysis of the hypergeometric distribution 
which produced Theorems \ref{thm:lpanalogue} and \ref{thm:talagrandanalogue}. We turn to the proofs of these bounds in the concluding section.

\begin{figure}
  \begin{subfigure}{\textwidth}
    \centering
    \includegraphics[width=\linewidth,height=7.5cm,keepaspectratio]{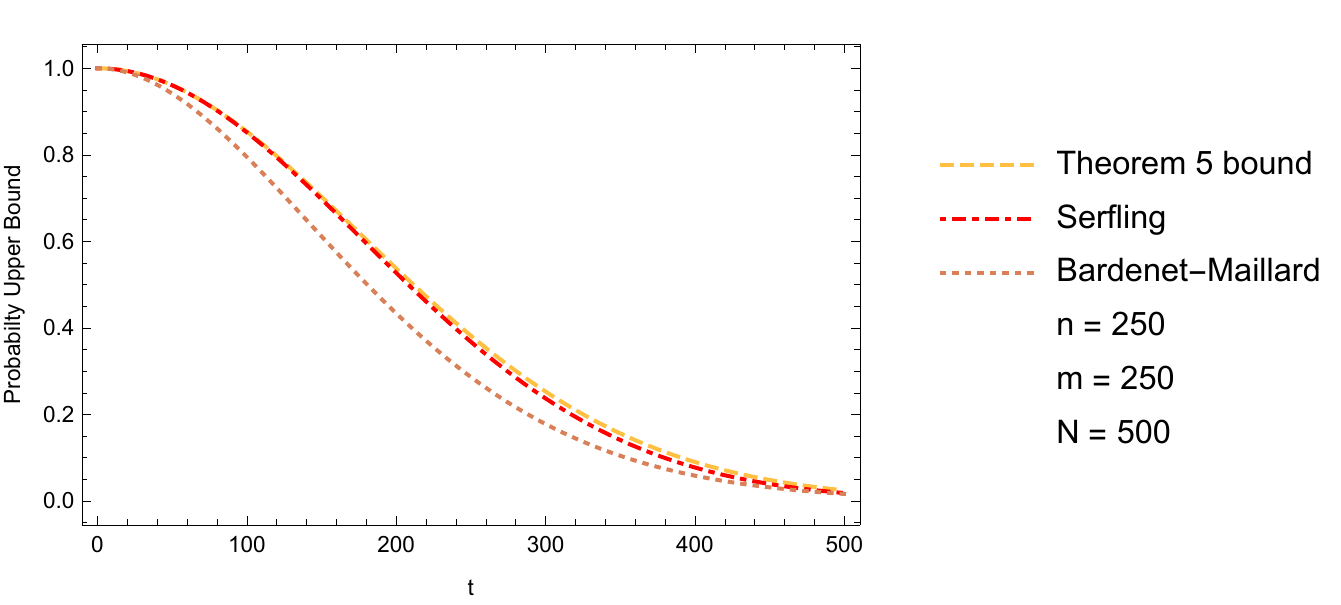}
    \caption{1a}
    \label{fig5:sfig1}
  \end{subfigure}\\
  \begin{subfigure}{\textwidth}
    \centering
    \includegraphics[width=\linewidth,height=7.5cm,keepaspectratio]{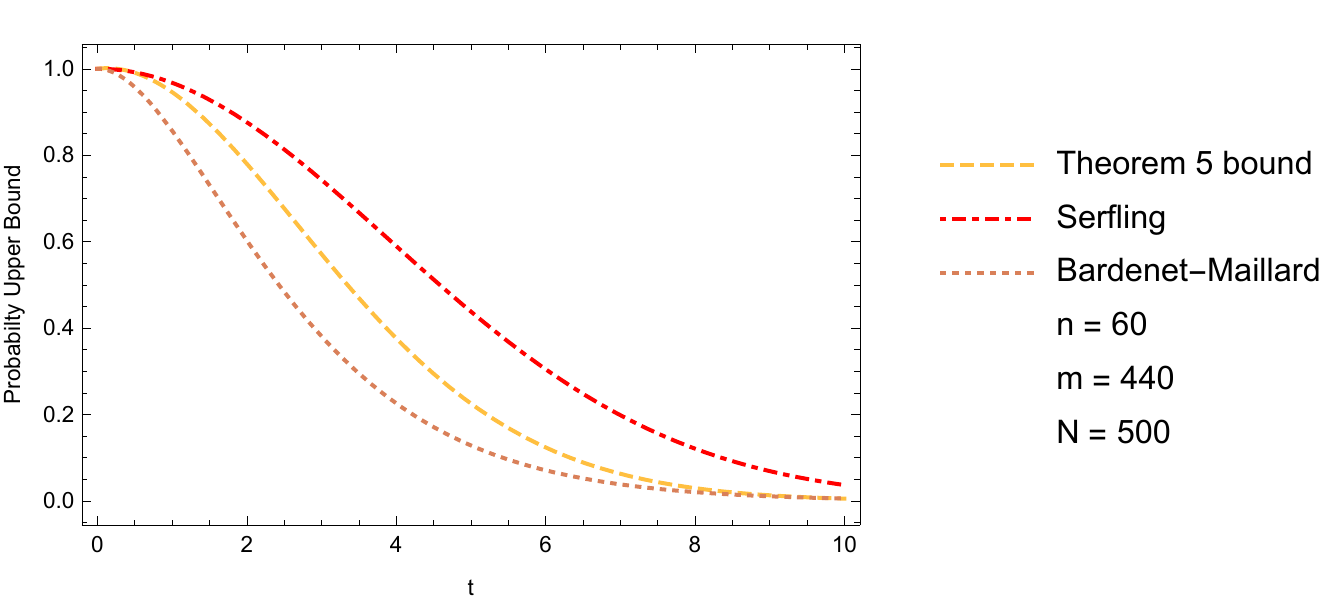}
    \caption{1b}
    \label{fig5:sfig2}
  \end{subfigure}
\caption{Comparison of the bounds of Theorem \ref{thm:genben} to Bardenet and Maillard's bound  \eqref{expbnd:bm2} and Serfling's bound. 
The first sub-figure \eqref{fig5:sfig1}, corresponds to the Wilcoxon example. 
The second sub-figure \eqref{fig5:sfig2}, corresponds to the Klotz example. 
} 
\label{fig:figure5}
\end{figure}

\section{Proofs of the Bounds}  
\label{sec:proofs}
Our proofs depend on a version of Stirling's formula from Robbins \citep{ROBBINS1}.
\begin{lem}
For $n \in \mathbb{N}_0$
\begin{align}
&\sqrt{2\pi n} \left(\frac{n}{e}\right)^n e^{\frac{1}{12n+1}}
\leq n! \leq
\sqrt{2\pi n}\left(\frac{n}{e}\right)^n e^{\frac{1}{12n}}\ . \label{STIRLINGS_FORMULA}
\end{align}
\end{lem}

To prove \eqref{HYPERGEOMETRIC:LP_BOUND}, we will need some additional tools. We start with the following lemma.
\begin{lem}
Suppose $S_n \sim \text{Hypergeometric}(n,D,N)$ with $1 \leq n < D \leq \lfloor N/2\rfloor$ and $1 \leq k \leq n-1$. 
Then for $k \geq n(D/N)$ we have
\begin{align}
P(S_n=k) &\leq \frac{1}{\sqrt{2\pi}} \frac{\sqrt{D(N-D) n (N-n)}}{ \sqrt{k(D-k)(n-k) (N-D-(n-k)) N}} \nonumber \\
&\ \ \ \cdot \exp\left(-\frac{2nN}{N-n} u^2\right) \exp\left(-\frac{n}{3}\left(1+\frac{n^3}{(N-n)^3}\right) u^4 \right)\ . 
\label{HYPERGEOMETRIC_DEVIATE_BOUND}
\end{align}
\end{lem}
\begin{proof}[{\bf Proof}]
The proof follows by direct analysis. Using
Stirling's formula \eqref{STIRLINGS_FORMULA}, we have
\begin{align}
 P(S_n=k) &=\frac{{D \choose k}{N-D \choose n-k}}{{N \choose n}} \nonumber \\
 &\leq \frac{1}{\sqrt{2\pi}} \frac{\sqrt{D(N-D) n (N-n)}}{ \sqrt{k(D-k)(n-k) (N-D-(n-k)) N}} \nonumber \\
 & \ \ \ \cdot \frac{D^D (N-D)^{N-D} n^n (N-n)^{N-n} }
                     { k^k (D-k)^{D-k} (n-k)^{n-k} (N-D - (n-k))^{N-D -(n-k)} N^N} \nonumber \\
 & \ \ \ \cdot \frac{\exp \left ( \frac{1}{12D} + \frac{1}{12(N-D)} + \frac{1}{12n} + \frac{1}{12(N-n)} \right ) } 
                            {\exp \left ( \frac{1}{12k+1} + \frac{1}{12 (D-k)+1} + \frac{1}{12 (n-k)+1} 
                              + \frac{1}{12(N-D-(n-k))+1} + \frac{1}{12 N+1} \right )} \nonumber \\
&=:  A \cdot B \cdot C\ .  \label{hgdecomposition}
\end{align}
We consider the $B$ term first. Define $u := k/n - \mu$, recalling $\mu := D/N$. We then have
\begin{align}
B &=  \frac{D^D (N-D)^{N-D} /N^N}{\left (\frac{k}{n} \right )^{k} \cdot \left ( 1 - \frac{k}{n} \right )^{n-k} 
              \cdot \left ( \frac{D-k}{N-n} \right )^{D-k} \left (1 - \frac{D-k}{N-n} \right )^{N-n-(D-k)} } \nonumber \\
&=  \frac{\left ( \frac{D}{N} \right )^D \left ( 1 - \frac{D}{N} \right )^{N-D}}
           { \left ( \frac{k}{n} \right )^k \left (1 - \frac{k}{n} \right )^{n-k} 
                \left ( \frac{D-k}{N-n} \right )^{D-k} \left (1 - \frac{D-k}{N-n} \right )^{N-n - (D-k)} } \nonumber \\
&= \frac{\left ( \frac{D}{N} \right )^D}{\left ( \frac{k}{n} \right )^{k} \left ( \frac{D-k}{N-n} \right )^{D-k}} \cdot 
          \frac{\left (1- \frac{D}{N} \right )^{N-D}}
                 {\left ( 1 - \frac{k}{n} \right )^{n-k} \left ( 1 - \frac{D-k}{N-n} \right )^{N-n-(D-k)} } \nonumber \\
&=  \frac{\left ( \frac{D}{N} \right )^k \left (\frac{D}{N} \right )^{D-k}}
                 {\left ( \frac{k}{n} \right )^{k} \left ( \frac{D-k}{N-n} \right )^{D-k}} \cdot 
          \frac{\left (1- \frac{D}{N} \right )^{n-k} \left (1- \frac{D}{N} \right )^{N-D- (n-k)}}
                 {\left ( 1 - \frac{k}{n} \right )^{n-k} \left ( 1 - \frac{D-k}{N-n} \right )^{N-n-(D-k)} } \nonumber \\
&=  \left ( \frac{\mu}{u+\mu} \right )^k \left ( \frac{\frac{N-n}{N}}{\frac{D-k}{D}} \right )^{D-k} \cdot 
          \left ( \frac{1-\mu}{1 - (u+\mu)} \right )^{n-k} \left ( \frac{\frac{N-n}{N}}{\frac{N-n-(D-k)}{N-D}} \right )^{N-D-(n-k)} \nonumber \\
&=  \exp ( - n \Psi (u, \mu)) \cdot \frac{\left ( \frac{N-n}{N} \right )^{N-n}}
                         {\left ( \frac{D-k}{D} \right )^{D-k} \left ( \frac{N-D - (n-k)}{N-D} \right )^{N-D - (n-k)} } \nonumber \\
&=   \exp ( - n \Psi (u, \mu)) \cdot  B_2 \nonumber
\end{align}
where the first factor corresponds to the same function as in Talagrand's argument for the binomial distribution \citep[pp.~48--50]{TALAGRAND1} and we recall
\[
\Psi(u,\mu) := (u+\mu) \log \left ( \frac{u+\mu}{\mu} \right ) + \left ( 1 - (u+\mu) \right ) \log \left ( \frac{1 - (u+\mu)}{1-\mu} \right ) .
\]
Now, we can further re-write $B_2$ as
\begin{align}
B_2 
=   \left ( \frac{\frac{N-n}{N}}{\frac{D-k}{D}} \right )^{D-k}  
             \cdot  \left ( \frac{\frac{N-n}{N}}{\frac{N-n-(D-k)}{N-D}} \right )^{N-D-(n-k)}     
=: \exp( - \Gamma ) \nonumber
\end{align}
where
\begin{align}
\Gamma &=  -\log (B_2) \nonumber  \\
&=  (D-k) \log \left[ \frac{\left(\frac{D-k}{D}\right)}{\left (\frac{N-n}{N} \right )} \right] 
+\left[ N-n - (D-k)\right] \log \left[\frac{\left ( \frac{N-n - (D-k)}{N-D} \right )}{\left ( \frac{N-n}{N} \right ) }\right] \nonumber \\
\nonumber \\
&=  (N-n) \left(\frac{D-k}{N-n}\right) \log \left[ \frac{(D-k)/D}{(N-n)/N} \right] 
+ \left(1 - \frac{D-k}{N-n}\right) \log \left[ \frac{[N-n - (D-k)]/(N-D)}{(N-n)/N} \right]\ .\nonumber  
\label{B2-equivalentForm}
\end{align} 
Now $k = n(u + \mu)$, so 
\begin{align}
\frac{D-k}{N} =  \mu - \frac{n}{N} (u +\mu) = \mu(1 -n/N) - (n/N) u\ . \nonumber
\end{align}
and 
\begin{align}
\frac{(D-k)/N}{(N-n)/N} = \frac{\mu(1-(n/N)) - (n/N) u}{1 - n/N} = \mu - \frac{n/N}{1-n/N} u\ . \nonumber
\end{align}
Thus we also have 
\begin{align}
1 - \frac{(D-k)/N}{(N-n)/N} = 1-\mu + \frac{n/N}{1-(n/N)} u\ . \nonumber
\end{align}
Thus it follows that, with $f = f_N := n/N$, $\overline{f} = \overline{f}_N := 1 - f_N$, 
\begin{align}
\frac{\Gamma}{N-n} 
&=  \left ( \mu - \frac{f}{\overline{f}}u \right ) \log \left[ \left ( \mu - \frac{f}{\overline{f}}u \right ) \frac{1}{\mu} \right] 
+ \  \left(1-\mu  + \frac{f}{\overline{f}}u \right ) \log \left[ \left (1- \mu +\frac{f}{\overline{f}}u \right ) \frac{1}{1-\mu} \right] 
\nonumber \\
\nonumber \\
&=  \Psi \left ( \frac{f}{\overline{f}} u , 1-\mu \right ) \nonumber
\end{align}
where $\Psi $ is as defined above.  Thus the $B$ term can be rewritten as 
\[
B = \exp \left ( -n\Psi ( u, \mu) - (N-n) \Psi \left ( \frac{f}{\overline{f}} u , 1-\mu \right ) \right )\ . 
\]
Now $\Psi$ satisfies $\Psi (0,\mu) = 0$, $\frac{\partial}{\partial u} \Psi (0,\mu) =0$, and, 
as in Talagrand (as well as  van der Vaart and Wellner \citep[pp.~460-461]{MR1385671}), 
\[
\frac{\partial^2}{\partial u^2} \Psi (u,\mu) = \frac{4}{1 - 4 (u - (1/2-\mu))^2} \ge 4 \left (1 + 4 ( u - (1/2-\mu))^2 \right )\ .
\]
Thus 
\begin{align}
&\frac{\partial^2}{\partial u^2} \left[ n \Psi (u,\mu) 
            + (N-n) \Psi \left ( \frac{f}{\overline{f}} u , 1-\mu \right ) \right] \nonumber \\
&=  n \frac{4}{1 - 4(u- (1/2-\mu))^2}
          + (N-n) \frac{4 (f /\overline{f})^2}{1 - 4 \left (\frac{f}{\overline{f}} u - (\mu-1/2) \right )^2} \nonumber \\
&\ge  4 n \left ( 1 + 4 (u - (1/2-\mu))^2 \right ) + 4 (N-n) (f / \overline{f} )^2  
                \left ( 1 + 4 \left (\frac{f}{\overline{f}} u - (\mu-1/2)\right )^2 \right )\ . \nonumber
\end{align}
Integration across this inequality yields 
\begin{align}
\frac{\partial}{\partial u} \left[ n \Psi (u,\mu) + (N-n) \Psi \left ( \frac{f}{\overline{f}} u , 1-\mu \right ) \right]
&\ge  4n \left ( u + \frac{1}{3} u^3 \right ) 
             + 4 (N-n) \left ( \frac{f}{\overline{f}} \right )^2   \left ( u + \frac{1}{3} \left ( \frac{f}{\overline{f}} \right )^2 u^3 \right ) 
              \nonumber \\
&=  4 \left ( n + (N-n) \left (\frac{n}{N-n} \right )^2 \right ) u \nonumber \\
& \ \ \ + \ \frac{4}{3} \left (n + (N-n) \left ( \frac{n}{N-n} \right )^4 \right ) u^3    \nonumber \\
&=  \frac{4 nN}{N-n} u  + \frac{4}{3} n \left ( 1 + \frac{n^3}{(N-n)^3} \right ) u^3\ .   \label{FirstDerivativeInequality}
\end{align}
Here we used
\begin{align} 
\int_0^u ( 1 + 4 (v - (1/2-\mu))^2) dv 
&=  u + \frac{4}{3} ( v - (1/2-\mu))^3 \bigg |_0^u \nonumber \\
&=  u + \frac{4}{3} \left[ u - (1/2 -\mu)^3 - (- (1/2-\mu))^3 \right] \nonumber \\
&=  u + \frac{4}{3} \left[ (u - (1/2-\mu))^3 + (1/2-\mu)^3 \right] \nonumber \\
&\ge  u + \frac{4}{3} \left[ (u/2)^3 + (u/2)^3 \right] \nonumber \\
&=  u + (1/3) u^3 \nonumber
\end{align}
where the inequality follows since the function $\beta \mapsto (u-\beta)^3 + \beta^3$ is minimized by 
$\beta = u/2$: with $h_u (\beta) \equiv (u-\beta)^3 + \beta^3$,
\begin{align}
h_u' (\beta) 
&=  3 (u-\beta)^2 (-1) + 3 \beta^2 = 3 \{ \beta^2 - (\beta^2 - 2u \beta + u^2)\} \nonumber \\
&=  3 u \{ 2\beta - u \} = 0 \ \ \ \mbox{if} \ \ \beta = u/2, \nonumber
\end{align}
while $h_u'' (\beta) = 6u > 0$.
Similarly,
\begin{align} 
\int_0^u \left ( 1 + 4 \left (\frac{f}{\overline{f}} v - (\mu-1/2) \right )^2 \right ) dv 
&=  u + \frac{4}{3}\left ( \frac{f}{\overline{f}} v - (\mu-1/2) \right )^3 \frac{\overline{f}}{f} \bigg |_0^u \nonumber \\
&=  u + \frac{4}{3}\frac{\overline{f}}{f}  \left[  \frac{f}{\overline{f}} u -  (\mu-1/2)^3 - \left (- (\mu-1/2) \right )^3 \right] \nonumber \\
&=  u + \frac{4}{3} \frac{\overline{f}}{f} \left[ \left ( \frac{f}{\overline{f}} u - (\mu-1/2) \right )^3 + (\mu-1/2)^3 \right] \nonumber \\
&\ge  u + \frac{4}{3} \frac{\overline{f}}{f} \left[ \left ( \frac{f u}{\overline{f} 2} \right )^3 +   \left ( \frac{f u}{\overline{f} 2} \right )^3 \right]
\nonumber\\
&=  u + \frac{1}{3} \left ( \frac{f}{\overline{f}} \right)^2 u^3\ . \nonumber
\end{align}
Integrating across \eqref{FirstDerivativeInequality} yields 
\begin{align}
n \Psi (u,\mu) + (N-n) \Psi \left ( \frac{f}{\overline{f}} u , 1-\mu \right ) 
\ge  \frac{2nN}{N-n} u^2  + (1/3) n \left ( 1 + \frac{n^3}{(N-n)^3} \right ) u^4\ . \nonumber
\end{align}
Thus the $B$ term in \eqref{hgdecomposition} has the following bound:
\begin{align}
B &\leq \exp\left(-\frac{2nN}{N-n} u^2\right) \exp\left(-\frac{n}{3}\left(1+\frac{n^3}{(N-n)^3}\right) u^4 \right)\ . 
\label{Bbound}
\end{align}
We next analyze the $C$ term in \eqref{hgdecomposition}. We have
\begin{align}
C &=
\frac{\exp \left ( \frac{1}{12D} + \frac{1}{12(N-D)} + \frac{1}{12n} + \frac{1}{12(N-n)} \right ) } 
                            {\exp \left ( \frac{1}{12k+1} + \frac{1}{12 (D-k)+1} + \frac{1}{12 (n-k)+1} 
                              + \frac{1}{12(N-D-(n-k))+1} + \frac{1}{12 N+1} \right )} \nonumber \\
&=
\exp\left(\frac{1}{12D}-\frac{1}{12 (D-k)+1}\right)
\exp\left(\frac{1}{12(N-D)}-\frac{1}{12(N-D-(n-k))+1}\right)\nonumber \\
&\ \cdot \exp\left(\frac{1}{12n}-\frac{1}{12k+1}\right)
\exp\left(\frac{1}{12(N-n)}-\frac{1}{12 (n-k)+1}\right)\exp\left(-\frac{1}{12 N+1}\right) \nonumber \\
 &=
 \exp\left(\frac{-12k+1}{[12D][12(D-k)+1]}\right)
 \exp\left(\frac{-12[n-k]+1}{[12(N-D)][12([N-D]-[n-k])+1]}\right) \nonumber\\
& \ \ \ \cdot
\exp\left(\frac{1-12 (n-k)}{12 (12 k+1) n}\right)
\exp\left(\frac{1-12(N-2n+k)}{12 (12(n-k)+1) (N-n)}\right)
\exp\left(-\frac{1}{12 N+1}\right) \nonumber \\
&\leq 1 \label{Cbound}
\end{align}
where the final inequality follows since $k \in [\lceil n\mu\rceil,\dots, n-1]$ and $n \leq D \leq \lfloor N/2\rfloor$ 
which implies that each exponential argument preceding the inequality is negative. 
This gives a bound of $1$ on the product.
As the $A$ term in \eqref{hgdecomposition} is already in the claimed form, 
combining \eqref{Bbound} and \eqref{Cbound} proves
the claim.
\end{proof}

Next we develop an upper bound for hypergeometric tail probabilities. This bound is similar to that discussed by Feller for the binomial \citep[pp.~150-151]{FELLER1}.
To our knowledge this result is new.

\begin{lem}\label{hgbinanalogue}
Suppose $S_{n,D,N} \sim Hypergeometric(n,D,N)$, $N>4$ and $1 \leq n, D \leq N-1$. For $k >  (nD)/N$,  we have
\begin{align}
P\left(S_{n,D,N} \geq k\right) \leq 
P\left(S_{n,D,N} = k\right) 
\left(\frac{k(N-D-n+k)}{Nk-nD}\right).\label{HYPERGEOMETRIC_TAIL_BOUND}
\end{align}
\end{lem}

\begin{proof}[{\bf Proof}]

Suppose first that $n \leq D$ and $k=n$. Then \eqref{HYPERGEOMETRIC_TAIL_BOUND} becomes
\[
P\left(S_{n,D,N} \geq n\right) \leq 
P\left(S_{n,D,N} = n\right) \left(\frac{n(N-D-n+n)}{Nn-nD}\right)
= 
P\left(S_{n,D,N} = n\right) \left(\frac{n(N-D)}{n(N-D)}\right)
=
P\left(S_{n,D,N} = n\right)\ .
\]
Since $P\left(S_{n,D,N} \geq n\right)=P\left(S_{n,D,N} = n\right)$, the result holds in this case.
Next, suppose $D < n$ and $k=D$. Then \eqref{HYPERGEOMETRIC_TAIL_BOUND} becomes
\[
P(S_{n,D,N}\geq D)\leq P(S_{n,D,N}=D)\left(\frac{D(N-D-n+D)}{ND-nD}\right)=P(S_{n,D,N}=D)\left(\frac{D(N-n)}{D(N-n)}\right) = P(S_{n,D,N}=D)\ .
\]
Since $P(S_{n,D,N}\geq D)=P(S_{n,D,N}=D)$, the result holds in this case too.

If $(n,D,N)$ is a population such that $\lfloor (nD)/N\rfloor +1 = n\wedge D$, we are done.
Supposing this is not the case, let $\lfloor (n D)/N\rfloor +1 \leq (j-1) < j \leq n\wedge D$.  
Assume the result holds when $k=j$. We will show this implies the result holds for $k=j-1$. We have
\begin{align}
P\left(S_{n,D,N} \geq j-1\right) 
&=
P\left(S_{n,D,N} = j-1\right) +
P\left(S_{n,D,N} \geq j\right)  \nonumber \\
&\leq
P\left(S_{n,D,N} = j-1\right) +
P\left(S_{n,D,N} = j\right)\left[\frac{j(N-D-n+j)}{Nj-nD}\right]   \text{(by induction hypothesis)} \nonumber \\
&=
P\left(S_{n,D,N} = j-1\right)
\left[1+\frac{P\left(S_{n,D,N} = j\right)}{P\left(S_{n,D,N} = j-1\right)}\left[\frac{j(N-D-n+j)}{Nj-nD}\right]\right] \nonumber \\
&=
P\left(S_{n,D,N} = j-1\right)
\left[1+\frac{(D-j+1)(n-j+1)}{j(N-D-n+j)}\left[\frac{j(N-D-n+j)}{Nj-nD}\right]\right] \nonumber \\
&=
P\left(S_{n,D,N} = j-1\right)
\left[1+\frac{(D-j+1)(n-j+1)}{Nj-nD}\right]\ . \nonumber
\end{align}
Under the current assumption, the right-hand side equals
\[
P\left(S_{n,D,N} = j-1\right) 
\left[\left(\frac{(j-1)(N-D-n+j-1)}{N(j-1)-nD}\right)\right]\,
\]
so we see it is enough to show
\[
\left[\left(\frac{(j-1)(N-D-n+j-1)}{N(j-1)-nD}\right)\right]-\left[1+\frac{(D-j+1)(n-j+1)}{Nj-nD}\right]\geq 0\ .
\]
Combining terms and simplifying, we find this equivalent to showing
\[
\frac{N (D-j+1)(n-j+1)}{(Nj-nD)(N(j-1)-nD)} \geq 0\ .
\]
Since we assume $\lfloor (n D)/N\rfloor +1 \leq (j-1) < j \leq n \wedge D$, 
we see that each term in parentheses  in the
fraction is non-negative. In particular, since $j \geq \lfloor (nD)/N \rfloor +2 > (nD)/N+1,$ we have
\[
N(j-1)-nD > N((nD)/N)-nD=0\ .
\]
Thus, the expression is non-negative. This implies the claim.
\end{proof}

We next prove a technical lemma.

\begin{lem}\label{HG_Technical_Lemma}
Fix $N > 4$. Suppose that $n < D \leq \lfloor N/2\rfloor$ and that $\gamma := (N-n)/n$. 
For all triples 
\[
(\mu,u,\gamma) \in \left[\frac{n+1}{N},\frac{1}{2}\right] \times \bigg (0,\frac{1}{2}\bigg] \times (1,\infty)
\]
we have
\begin{align}
\frac{\mu(1-\mu)(u+\mu)(\gamma(1-\mu)+u)}{(1-u-\mu)(\gamma\mu-u)}
\leq
\frac{1}{4}\frac{(u+(1/2))(\gamma(1-(1/2))+u)}{(1-u-(1/2))(\gamma(1/2)-u)}\ .\label{A_TERM_FUNCTION_BOUND}
\end{align}
\end{lem}

We pause to outline the strategy used to prove this statement, since the proof requires a rather detailed algebraic argument.
We break the quantity into two functions, $f$ and $g$, 
the second of which, $g$, is parabolic on $\mu \in [(n+1)/N,1/2]$. We demonstrate
that $f$ is maximized at $\mu=1/2$. We do this  by obtaining the only root which falls in the interval, determining that it yields a local minimum, and finally
showing the function is larger at the upper boundary of $\mu=1/2$. 

We then show that $g$ has a local maximum in the interior of the interval (for $0<u<1/2$). 
Using the quadratic $g$ function as a scaling function,
we then define an upper envelope to the function of interest in terms of $f$, along with a second function that agrees with the function of interest at $\mu=1/2$. 
By defining the two new functions in terms of $f$ (scaled by positive numbers, which are obtained at fixed-points of $g$),
we are still able to claim these functions are maximized at $\mu=1/2$.

We then demonstrate the function of interest increases monotonically between the value of $\mu$ where it intersects its envelope and $\mu=1/2$. 
We finally show that at the right endpoint of $\mu=1/2$, the quantity of interest exceeds
its envelope at the left end-point. This will prove the claim; the details now follow.

\begin{proof}[{\bf Proof of Lemma~\ref{HG_Technical_Lemma}}]
With the previous comments in mind, define the following functions:
\begin{align}
f(\mu) &:=  \frac{\mu\left(1-\mu\right)}{\left(1-u-\mu\right)(\gamma\mu-u)} \nonumber \\
\text{ and } \ \ g(\mu) &:= \left(u+\mu\right)\left(\gamma(1-\mu)+u\right)\ . \label{twofunctions}
\end{align}
Note that the product $f(\mu)g(\mu)$ gives the quantity on the left-hand side of \eqref{A_TERM_FUNCTION_BOUND}. We first analyze $f(\mu)$.  Taking its derivative, we find
\[
f'(\mu) = 
\frac{u ( (\gamma -1) \mu^2 +2\mu(1-u)-(1-u))}{(1-u -\mu)^2 (u-\gamma  \mu )^2}\ .
\]
Seeking critical points, we find $f'(\mu)$ has the following roots:
\[
\frac{\pm \sqrt{(1-u)(\gamma-u)}+u-1}{\gamma -1}\ .
\]
Since $\mu \in (0,1/2)$, only the positive root is of potential interest. Since $\gamma > 1$ under the current restrictions, we have
\[
\frac{ \sqrt{(1-u)(\gamma-u)}+u-1}{\gamma -1} \geq \frac{ \sqrt{(1-u)^2}+u-1}{\gamma -1} = 0\ .
\] 
Additionally, we can see
\[
\frac{\sqrt{(1-u)(\gamma-u)}+u-1}{\gamma -1} \leq \frac{1}{2}  
\]
since, after algebra, it is equivalent to showing
\[
0 \leq
\frac{(\gamma-1)^2}{4}  
\]
which follows under the assumptions. A similar argument shows that the corresponding root with the negative radical is always negative, and therefore does not affect the current investigation.
Next, differentiate again and evaluate the second derivative at the root. We then find
\begin{align}
f''\left(\mu\right)\ \bigg|_{ \left(\frac{\sqrt{(1-u)(\gamma-u)}+u-1}{\gamma -1}\right)} 
&= \frac{\left[2 (\gamma -1)^4 (1-u) u (\gamma-u)\right] \left[\left(\gamma ^2+1\right) u+2 \gamma  \sqrt{(1-u) (\gamma-u)}-\gamma^2 -\gamma \right]}
{\left[\sqrt{(1-u) (\gamma-u)}-\gamma  (1-u)\right]^3\left[\gamma  \left(\sqrt{(1-u) (\gamma-u)}-1\right)+u\right]^3}  \nonumber \\
&=: \frac{[a(u,\gamma)][b(u,\gamma)]}{[c(u,\gamma)]^3[d(u,\gamma)]^3}\ . \nonumber
\end{align}
We next show that this quantity is positive for any $(u,\gamma) \in (0,1/2)\times (1,\infty)$.  
It is clear that $a(u,\gamma)$ is always positive under the current assumption, since each term in the product is positive.

We next claim $b(u,\gamma) < 0$ for all $(u,\gamma) \in (0,1/2)\times(1,\infty)$.
This claim is equivalent to showing
\[
2\gamma  \sqrt{(1-u) (\gamma-u)}< \gamma^2(1-u)+(\gamma-u)\ .
\]
Since both sides are positive, we square both sides and simplify to find that the claim is equivalent to showing
\[
0 < (\gamma -1)^2 (-\gamma +\gamma  u+u)^2\ .
\]
As this last claim follows for any admissible pair, we conclude $b(u,\gamma) < 0$ for all $(u,\gamma) \in (0,1/2)\times(1,\infty)$.

We next show that $c(u,\gamma) < 0$ for all $(u,\gamma) \in (0,1/2)\times(1,\infty)$. This claim is equivalent to
\[
(1-u)(\gamma-u) < \gamma^2(1-u)^2
\]
which, after expanding and re-arranging, is equivalent to the claim
\[
0 < (\gamma -1) (1-u) (\gamma -\gamma  u-u)
\] 
for all $(u,\gamma) \in (0,1/2)\times(1,\infty)$. On this set, it is clear $\gamma-1$ and $1-u$ are positive for any admissible pair. Hence, we need only show $(\gamma-\gamma u-u)>0$
on this set. But this is equivalent to claiming $\gamma(1-u)>u$ for any pair on this set, which is true because $\gamma > 1$ and $u < 1/2$. Thus we conclude $c(u,\gamma)<0$ for all
$(u,\gamma) \in (0,1/2)\times(1,\infty)$.

We finish this sub-argument by showing $d(u,\gamma) > 0$ for $(u,\gamma) \in (0,1/2)\times(1,\infty)$. This claim is equivalent to
\[
\gamma \sqrt{(1-u) (\gamma-u)} > \gamma-u
\]
for all admissible pairs. Since both sides are positive, we square and simplify to find the claim equivalent to
\[
p(u):=\gamma^2-\gamma^2u-\gamma+u>0\ .
\]
Viewing the left-hand side as a function of $u$, we differentiate to see $p'(u)=1-\gamma^2<0$ for any choice of $\gamma >1$. So, $p(u)$ decreases in $u$ for any $\gamma >1$
Hence
\[
p(u) > \gamma^2-\frac{\gamma^2}{2}-\gamma+\frac{1}{2} = \frac{\gamma^2-2\gamma+1}{2} = \frac{(\gamma-1)^2}{2} > 0\ .
\]
Thus we conclude $d(u,\gamma) > 0$ for $(u,\gamma) \in (0,1/2)\times(1,\infty)$.

To summarize: we have shown that for all $(u,\gamma) \in (0,1/2)\times(1,\infty)$, $a(u,\gamma) > 0$, $b(u,\gamma)<0$, $c(u,\gamma)<0$ and $d(u,\gamma) > 0$.
This means that
\[
f''\left(\mu\right)\ \bigg|_{ \left(\frac{\sqrt{(1-u)(\gamma-u)}+u-1}{\gamma -1}\right)}  = \frac{[a(u,\gamma)][b(u,\gamma)]}{[c(u,\gamma)]^3[d(u,\gamma)]^3} > 0\ .
\]
Therefore we have found a local minimum of $f(\mu)$ that falls in $\left[(n+1)/N,1/2\right]$. Therefore, the maximum must be achieved at one of the endpoints.

We next show that the maximum is in fact achieved at $\mu = 1/2$. To do this, we compare the difference. Plugging in the definition $\gamma = (N-n)/n$, and simplifying, we find:
\[
f\left(\frac{1}{2}\right)-f\left(\frac{n+1}{N}\right) = 
\frac{n u (N-2n-2) (n u (N-2n-2)+N)}
{(1-2 u)(N(1-u)-n-1)(N-2nu-n)((n+1) (N-n)-n N u)}\ .
\]
Each term in this expression is positive for all $u\in(0,1/2)$ and hence the entire expression is positive. To see this, first observe that
the restriction $n < D \leq \lfloor N/2\rfloor$ means that the maximum value $n$ can attain is $\lfloor N/2\rfloor-1$.  This implies
$N-2n-2 \geq 0$. Since we also restrict $u\in(0,1/2)$, we also have $(N(1-u)-n-1)\geq 0$ and $(N-2nu-n)\geq 0$. Finally note
\[
(n+1) (N-n)-n N u \geq (n+1) (N-n)-\frac{n N}{2} = \frac{n (N-2 n-2)}{2} +N \geq 0\ .
\]
We conclude that $f(\mu)$ is maximized at $\mu=1/2$ over all choice of $(u,\gamma) \in (0,1/2)\times (1,\infty)$.

We next consider the function $g(\mu)$, defined in \eqref{twofunctions}. We write it again, its first two derivatives, and its critical point $\mu^*$ for subsequent discussion.
As this function is much simpler than $f(\mu)$, we present these quantities without comment.
\begin{align}
g(\mu) &= \left(u+\mu\right)\left(\gamma(1-\mu)+u\right)\ ,\nonumber \\
g'(\mu) &= -2 \gamma  \mu +\gamma -\gamma  u+u\ , \nonumber \\
g''(\mu) &= -2 \gamma\ , \nonumber \\
\text{ and } \ \mu^* &= \frac{\gamma(1 - u)+u}{2 \gamma}\ . \nonumber 
\end{align}

Since $g''(\mu) < 0$ for any choice of $(u,\gamma) \in (0,1/2)\times(1,\infty)$, we see that $\mu^*$ is a local maximum. For any $\gamma > 1$, we also see the critical point
decreases for $u \in (0,1/2)$, from a value of $1/2$ at $u=0$ to a value of $(1/4)+(1/(4\gamma))$. As $\gamma \nearrow \infty$, this approaches $1/4$ asymptotically. Hence
for any $(u,\gamma) \in (0,1/2)\times(1,\infty)$, the maximum of the function is attained for $\mu\in(0,1/2)$. Since we are ultimately interested in understanding the product
$f(\mu)g(\mu)$, we next show that the maximum of $g$ occurs at a value greater than the local minimum of $f$. We do this by comparing their difference to zero.
The claim
\[
\left[\frac{\gamma(1 -u)+u}{2 \gamma}\right]- \left[\frac{\sqrt{(1-u)(\gamma-u)}+u-1}{\gamma -1}\right] > 0
\]
is equivalent to the claim
\[
(\gamma-1)(\gamma(1-u)+u)+2\gamma(1-u) > 2\gamma\sqrt{(1-u)(\gamma-u)}\ .
\]
Both sides of this inequality are positive. So, we square them and simplify to find that the claim is equivalent to the claim
\[
(\gamma -1)^2 (\gamma(1 - u)-u)^2 > 0\ .
\]
The claim follows by the final form, since the square each quantity positive. We now define three related functions.
\begin{align}
ue(\mu) &:= g\left(\frac{\gamma(1 -u)+u}{2 \gamma}\right)f(\mu) = \frac{(1-\mu) \mu  (\gamma +\gamma  u+u)^2}{4 \gamma  (1-\mu -u) (\gamma  \mu -u)}\ , \nonumber \\
t(\mu) &:= g(\mu)f(\mu) =
\frac{\mu(1-\mu)(u+\mu)(\gamma(1-\mu)+u)}{(1-u-\mu)(\gamma\mu-u)}\ ,
\nonumber \\
\text{ and } ep(\mu) &:= g(1/2)f(\mu) = \frac{(1-\mu) \mu  (1+2u) (\gamma +2u)}{4 (1-\mu -u) (\gamma  \mu-u )}\ .    
\nonumber 
\end{align}
First notice that $t(\mu)$ is the quantity of interest, which we wish to show is maximized at $\mu=1/2$.
As defined, the function $ue(\mu)$ is an upper envelope of $t(\mu)$, with agreement at $\mu=(\gamma(1-u)+u)/(2\gamma)$. 
$ep(\mu)$ is defined so that $ep(1/2)=t(1/2)$, that is $ep$ agrees with $t$ at the end-point of the $\mu$-interval. 
Consider the behavior of $t(\mu)$ on $\mu \in [(\gamma(1-u)+u)/(2\gamma),1/2]$. We have
\begin{align}
t'(\mu) &= 1-2 \mu + \frac{(\gamma +1) u^2 
(\gamma  \mu ^2-\mu ^2+2 \mu -2 \mu  u+u-1)
}{(1-\mu -u)^2 (\gamma  \mu -u)^2}\ . \label{t_positivity_argument} 
\end{align}
Since $\mu \leq 1/2$, the sign of $t'(\mu)$ may be determined by the behavior of the third term in the numerator. We consider its behavior separately. Let
\begin{align}
a(\mu) &:= \gamma  \mu ^2-\mu ^2+2 \mu -2 \mu  u+u-1\ , \nonumber \\
a'(\mu) &= 2(1-u+\mu(\gamma-1))  > 0\ , \nonumber \\
\text{ and }\ \ a\left(\frac{\gamma(1-u)+u}{2\gamma}\right) &=  \frac{(\gamma-1)(u-\gamma(1-u))^2}{4\gamma^2} > 0\ . \nonumber 
\end{align}
We see then that $a(\mu)$ will be non-negative for $\mu \in [(\gamma(1-u)+u)/(2\gamma),1/2]$. Therefore, $t'(\mu) > 0$ on the same interval. Hence, $t(\mu)$ is increasing
on the same interval. Finally, consider the difference 
\begin{align}
&ep\left(1/2\right) - ue\left(\frac{n+1}{N}\right)\nonumber \\
&=
\frac{N u 
\begin{pmatrix}
2 u^2 (N-2 n) (N-2 n-1)(2 n (N-n-1)+N)\\
+N u (N-n) (N-2 n-3)+(N-n)^2 (N-2 n-2)-4 n N u^3 (N-2 n-1)\\
\end{pmatrix}}
{4(1-2 u) (N-n) (N-n-2 n u) (N(1-u)-n-1) (N-n+nN(1-u)-n^2)}
\end{align}
where we have again substituted the definition $\gamma = (N-n)/n$. We will now argue that this quantity is positive for all $n \in \left\{1,\dots,\lfloor N/2\rfloor-2\right\}$.
This is sufficient to demonstrate $t(\mu)$ is maximized at $\mu=1/2$, since we are supposing $n < D \leq \lfloor N/2\rfloor$. This restriction is necessary to handle the 
sign-change implicit in the term $(N-2 n-3)$. There, for $n=\lfloor N/2\rfloor-2$ it equals (for integer values of $N/2$) 1, while it flips signs for $N/2-1$. However, this sign-change
is not problematic since our assumptions imply at $n=N/2-1$ that $D=N/2$, which is the value we are trying to demonstrate maximizes $t(\mu)$. 

We will demonstrate positivity by analyzing the terms in the expression. 
For simplicity, we will assume $N/2$ is an integer, though the same analysis will hold for odd values of $N$. 
We will consider some of the denominator terms first. We have, using the assumptions,
\[
(N-n)+(nN(1-u)-n^2) \geq 
\frac{n (N-2 n-2)}{2} +N > 0\ .
\]
We also have
\[
N-n-2 n u \geq N-2n \geq N-N+4 > 0\ .
\]
So we see all terms in the denominator are positive for any choice of $(u,n)$.
Hence, it is enough to show that under our assumptions
\[
z(u):= 2 u^2 (N-2 n) (N-2 n-1)(2 n (N-n-1)+N)+N u (N-n) (N-2 n-3)-4 n N u^3 (N-2 n-1) \geq 0\ .
\]
First viewing the left-hand-side as a function of $u$, we observe the following computations:
\begin{align}
z'(u) &= 4 u (N-2 n) (N-2 n-1)(2 n (N-n-1)+N)+N  (N-n) (N-2 n-3)-12 n N u^2 (N-2 n-1)\ ,  \nonumber \\
z''(u) &= 4 (N-2 n) (N-2 n-1)(2 n (N-n-1)+N)-24 n N u (N-2 n-1)\ ,  \nonumber \\
z'''(u) &= -24 n N  (N-2 n-1) \leq 0\ . \nonumber 
\end{align}
From the third derivative, we see $z''(u)$ is decreasing in $u$. Since 
$z''(0)=4 (N-2 n-1) (N+2 n (N-n-1)) (N-2 n) > 0$, we calculate the value of the second derivative at $u=1/2$ to find
\[
z''(u)\bigg |_{u=1/2} = 4 (N-2 n-1) (4 n^3+4 n^2+2 n N^2+N^2-6 n^2N-7 n N) =: 4(N-2n-1)\phi(n)\ ,
\]
where we define the function $\phi(n)$ in-line. We analyze the sign of $\phi(n)$ for $n \in \left\{1,\dots,\lfloor N/2\rfloor -2\right\}$. 
Treating $n$ as continuous temporarily, we differentiate twice to find
\begin{align}
\phi''(n) &:= 24n-12N+8\ . \nonumber 
\end{align}
Since we assume $n \in \left\{1,\dots,\lfloor N/2\rfloor -2\right\}$,  we see
\[
\phi''(n) = 24n-12N+8 \leq 24\left(\frac{N}{2}-2\right) -12N+8 = -40 \leq 0\ .
\] 
This implies $\phi(n)$ is concave in $n$. Evaluating at the admissible endpoints, we find
\begin{align}
\phi(1) &= 8+N(3N-13)\ , \nonumber \\
\text{ and } \ \ \ \phi((N/2)-2)&= \frac{N^2+12 N-32}{2}\ . \nonumber
\end{align}
For $N \geq 4$, both of these expressions are positive. By concavity we conclude $\phi(n) \geq 0$.
Therefore, we have that
\[
z''(u)\bigg |_{u=1/2} > 0\ ,
\]
and so we conclude $z''(u)>0$ for all $u \in (0,1/2]$. But since 
\[
z'(u)\bigg |_{u=0} = N  (N-n) (N-2 n-3) > 0\ ,
\]
we infer that $z'(u)>0$ for all $u \in (0,1/2]$. Finally, since $z(0) = 0$, we conclude that $z(u) > 0$ for all $u \in (0,1/2]$. But this implies that
\begin{align}
ep\left(1/2\right) - ue\left(\frac{n+1}{N}\right) > 0\ \ \ . \label{positivity_demonstration}
\end{align}
Therefore, we can define the following function
\[
\text{maj}(\mu) :=
\begin{cases}
ue(\mu) \text{ if } \mu \in \left[\frac{n+1}{N}, \frac{\gamma(1-u)+u}{2\gamma}\right] \\
\ \\
t(\mu) \text{ if } \mu \in \bigg(\frac{\gamma(1-u)+u}{2\gamma}, \frac{1}{2} \bigg ]\ .\\
\end{cases}
\]
Observe that for all $\mu \in [(n+1)/N,1/2]$, we have $\text{maj}(\mu) \geq t(\mu)$. Additionally, we know that $\text{maj}(\mu)$ is maximized at
$\mu = 1/2\ $: the argument following \eqref{t_positivity_argument} 
shows for $\mu$ such that $\text{maj}(\mu)=t(\mu)$, $\text{maj}(\mu)$
strictly increases; the argument following \eqref{positivity_demonstration}
shows $\text{maj}(\mu)$ increases to its maximum on the interval. 
Finally, since we know $\text{maj}(1/2) = ep(1/2) = t(1/2)$, we conclude $t(\mu)$ is maximized at $\mu=1/2$ for all choice of $(u,\gamma)\in (0,1/2]\times(1,\infty)$.
This completes the proof.
\end{proof}

We are now ready to prove \eqref{HYPERGEOMETRIC:LP_BOUND}.
\begin{proof}[{\bf Proof of Theorem~\ref{thm:lpanalogue}}]
Pick $\lambda,k > 0$ such that $k = \sqrt{n}\lambda+n\mu$, $k \geq n(D/N)$. We then have
\begin{flalign}
  &&
  &P\left(\sqrt{n}(\bar{X} - \mu) \geq \lambda\right)  
  &&\nonumber \\
  &&
  &=
  P\left(\sum_{i=1}^nX_i  \geq k\right) 
  &&\nonumber \\
  &&
  &\leq
  \left[P\left(\sum_{i=1}^nX_i = k\right)\right] 
  \left(\frac{k(N-D-n+k)}{Nk-nD}\right)\ 
  &&
  \text{by \eqref{HYPERGEOMETRIC_TAIL_BOUND}} \nonumber \\
  & 
  &&
  \leq
  \left[\frac{1}{\sqrt{2\pi}} \frac{\sqrt{D(N-D) n (N-n)}}{ \sqrt{k(D-k)(n-k) (N-D-(n-k)) N}} \left(\frac{k(N-D-n+k)}{Nk-nD}\right)\right]  
  &&
  \nonumber\\
  &&
  &\ \cdot\ \left[\exp\left(-\frac{2nN}{N-n} u^2\right) \exp\left(-\frac{n}{3}\left(1+\frac{n^3}{(N-n)^3}\right) u^4 \right)\right] 
  &&
  \text{by \eqref{HYPERGEOMETRIC_DEVIATE_BOUND}} \nonumber \\
  &&
  &= \left[A\right] \cdot \left[\exp\left(-\frac{2nN}{N-n} u^2\right) \exp\left(-\frac{n}{3}\left(1+\frac{n^3}{(N-n)^3}\right) u^4 \right)\right]  \ .
  &&
  \label{approxdeparture}
\end{flalign}
Recall that $u:=(k/n)-(D/N)$ in the previous bound.
Define $f:=n/N$, $\overline{f}:= 1-f_N = (N-n)/N$, $\mu:=D/N$, and 
Furthermore, define the ratio 
\[ \gamma := \frac{\overline{f}}{f} = \frac{N-n}{n}\ . \] 
We may then write:
\[
D-k 
= \left(N\frac{D}{N}-n\frac{k}{n}\right) 
= n\left(\frac{N}{n}\mu-\frac{k}{n}\right) 
= n\left(\frac{N}{n}\mu-u-\mu\right) 
= n\left(\gamma\mu-u\right)\ .
\]
Similarly we have
\[
N-n-(D-k) = N-n-n\left(\gamma\mu-u\right) = n(\gamma-\gamma\mu+u) = n(\gamma[1-\mu]+u)\ .
\]
Using these parametrizations, we may write
\begin{align}
\left[A\right]
&= 
\frac{1}{\sqrt{2\pi}} \sqrt{\frac{D(N-D) n (N-n)}{k(D-k)(n-k) (N-D-(n-k)) N}} \left(\frac{k(N-n-(D-k))}{Nn((k/n)-(D/N))}\right)\nonumber \\
&= 
\frac{1}{\sqrt{2\pi}} \sqrt{\frac{D(N-D) n (N-n) k^2(N-n-(D-k))^2}{k(D-k)(n-k) (N-n-(D-k)) N^3n^2}} \left(\frac{1}{u}\right)\nonumber \\
&= 
\sqrt{\frac{(N-n)}{2\pi n N u^2}} \sqrt{\frac{D}{N}\left(1-\frac{D}{N}\right)\frac{\left(\frac{k}{n}\right)}{\left(1-\frac{k}{n}\right)}\frac{(N-n-(D-k))}{(D-k)}} \nonumber \\
&= 
\sqrt{\frac{(N-n)}{2\pi n N u^2}} \sqrt{\mu\left(1-\mu\right)\frac{\left(u+\mu\right)}{\left(1-u-\mu\right)}\frac{\gamma(1-\mu)+u}{\gamma\mu-u}} \nonumber \\
&\leq
\sqrt{\frac{(N-n)}{2\pi n N u^2}} 
\sqrt{\frac{1}{4}\frac{(u+(1/2))(\gamma(1-(1/2))+u)}{(1-u-(1/2))(\gamma(1/2)-u)}} \label{New_Abound}
\end{align}
with the last inequality following by \eqref{A_TERM_FUNCTION_BOUND} established in Lemma \ref{HG_Technical_Lemma}. Observe under these parametrizations $u = \lambda/\sqrt{n}$. Hence, if we use \eqref{New_Abound} to provide an
upper bound for \eqref{approxdeparture}, substitute $\lambda/\sqrt{n}$ for $u$, and then simplify, the claim is proved.
\end{proof}

Some of the machinery developed in the preceding lemmas will be adapted to prove \eqref{TAL:HYPER:BOUND}. The argument follows.

\begin{proof}[{\bf Proof of Theorem~\ref{thm:talagrandanalogue}}]
Suppose now that $1 \leq n < D \leq N-1$. We consider $k$ such that $0 \vee n+D-N < k \leq n$. The decomposition of a Hypergeometirc probability into $A$, $B$, and $C$ terms stated in
\eqref{hgdecomposition} still applies. For $k \geq n (D/N)$, the bound on the $B$ term in \eqref{Bbound} still holds. Thus we may write
\begin{align}
B 
&\leq \exp \left ( - \frac{2nN}{N-n} u^2  \right ) \exp \left ( - \frac{n}{3} \left ( 1 + \frac{n^3}{(N-n)^3} \right ) u^4 \right ) \nonumber \\
&=
\exp\left(-\frac{2n}{1-\frac{n}{N}}  u^2  \right) \exp \left (-\frac{n}{4}u^4\right)\exp \left (-\frac{n}{12}u^4\right)\exp\left( -\left[\frac{n^4}{3(N-n)^3}\right]u^4\right)\ .
\label{Bbound2}
\end{align}
Also recall we showed that $C \leq 1$ at \eqref{Cbound} when $n \leq D \leq N/2$. In fact, the expression at \eqref{Cbound} shows $C \leq 1$ under the current assumptions. When $n \leq N/2$,
all exponential arguments may be determined to be negative by inspection. When $n > N/2$, the only fraction whose sign is unclear is
\[
\frac{1-12(N-2n+k)}{12 (12(n-k)+1) (N-n)}\ .
\]
However, this remains negative under the current assumptions since $n>N/2$ implies $k \geq n+D-N$. Therefore, $N+k \geq n+D$ and so $N+k-2n \geq D-N \geq 0$. 
We thus conclude $C \leq 1$.
Here though, we provide a new analysis of the $A$ term under the current assumptions.

\textit{\textbf{Case 1}}

First restrict $k$ so that $\mu_0 < \frac{k}{n} < 1-\frac{\mu_0}{2}$. We then have
\begin{align}
A  &=  \frac{1}{\sqrt{2\pi}} \sqrt{\frac{D(N-D)n(N-n)}{k(D-k)(n-k) (N-D-(n-k)) N}} \nonumber \\
  &=  \frac{N}{\sqrt{2\pi n}} \sqrt{\frac{\frac{D}{N}(1-\frac{D}{N})(1-\frac{n}{N})}
                        {\frac{k}{n}(D-k)(1-\frac{k}{n}) (N-D-n+k)}} \nonumber \\
  &\leq  \frac{N}{\sqrt{2\pi n}} \sqrt{\frac{(1/4)(1-\psi_0)}
                                                     {\mu_0(D-n+n\frac{\mu_0}{2})(\frac{\mu_0}{2}) (N-D-n+n\mu_0)}} 
  \leq  \frac{N}{\sqrt{2\pi n}} \sqrt{\frac{(1/4)(1-\psi_0)}{\mu_0(n\frac{\mu_0}{2})(\frac{\mu_0}{2}) (n\mu_0)}} \nonumber \\
  &= \frac{1}{\frac{n}{N}\sqrt{n}} \sqrt{\frac{2(1/4)(1-\psi_0)}{\pi \mu_0^4}} 
        \leq \frac{1}{\sqrt{n}} \frac{\sqrt{(1-\psi_0)}}{\psi_0\sqrt{2\pi \mu_0^4}}  .\label{Abnd:case1}
\end{align}
Combining \eqref{Abnd:case1} with \eqref{Bbound2} and \eqref{Cbound}, we have the bound
\[
\frac{{D \choose k}{N-D \choose n-k}}{{N \choose n} } 
\leq
\frac{K_{c1}}{\sqrt{n}} 
\exp\left(-\frac{2n}{1-\frac{n}{N}}  u^2  \right) 
\exp \left (-\frac{n}{4}u^4\right)\exp \left (-\frac{n}{12}u^4\right)
\exp\left( -\left[\frac{n^4}{3(N-n)^3}\right]u^4\right)
\]
where
\[
K_{c1} = \left[\frac{\sqrt{(1-\psi_0)}}{\psi_0\sqrt{2\pi \mu_0^4}}\right].
\]

\textit{\textbf{Case 2}}

Next, suppose that $1-\frac{\mu_0}{2}\leq\frac{k}{n} < 1$. This implies that 
\[
u =\frac{k}{n}-\frac{D}{N} \geq 1-\frac{\mu_0}{2} - \frac{D}{N} \geq 1-\frac{\mu_0}{2} - (1-\mu_0) = \frac{\mu_0}{2} .
\]
We can bound the $A$ term by
\begin{align}
  A  &=  \frac{1}{\sqrt{2\pi}} \sqrt{\frac{D(N-D)n(N-n)}{k(D-k)(n-k) (N-D-(n-k)) N}} \nonumber \\
  &=  \frac{N}{\sqrt{2\pi}} \sqrt{\frac{\frac{D}{N}(1-\frac{D}{N})(1-\frac{n}{N})}{\frac{k}{n}(D-k)(n-k) (N-D-n+k) }} \nonumber \\
  &\leq  \frac{N}{\sqrt{2\pi}} \sqrt{\frac{(1/4)(1-\psi_0)}
  {(1-\frac{\mu_0}{2})(D-n+1)(n-n+1) (N-D-n+n(1-\frac{\mu_0}{2})) }} \nonumber \\
  &\leq  \frac{N}{\sqrt{2\pi}} \sqrt{\frac{(1/4)(1-\psi_0)}{(1-\frac{\mu_0}{2})(n(1-\frac{\mu_0}{2})) }} 
             =  \frac{n\frac{N}{n}}{\sqrt{n}} \sqrt{\frac{(1/4)(1-\psi_0)}{2\pi(1-\frac{\mu_0}{2})^2}} \nonumber \\
  &\leq  \frac{n\frac{1}{\psi_0}}{\sqrt{n}} \sqrt{\frac{(1/4)(1-\psi_0)}{ 2\pi(1-\frac{\mu_0}{2})^2}} 
             =  \frac{n}{\sqrt{n}} \sqrt{\frac{(1/4)(1-\psi_0)}{ 2\pi\psi_0^2(1-\frac{\mu_0}{2})^2}} \ . \nonumber 
\end{align}
Taking the $\exp \left (-\frac{n}{12}u^4\right)$ term from \eqref{Bbound2} we have
\[
n\exp \left (-\frac{n}{12}u^4\right) \leq 
n\exp \left (-\frac{n}{12}\left(\frac{\mu_0}{2}\right)^4\right)  
=
n\exp \left (-\frac{\mu_0^4}{192}n\right) \  .
\]
This is maximized at
\[
n=\frac{192}{\mu_0^4},
\]
and so
\[
n\exp \left (-\frac{m}{12}u^4\right) \leq  \frac{192}{\mu_0^4e} .
\]
Combining the remaining terms in \eqref{Bbound} together with this bound of the A term and the C bound of 1 yields
\[
\frac{{D \choose k}{N-D \choose n-k}}{{N \choose n}} 
\leq
\frac{K_{c2}}{\sqrt{n}} 
\exp\left(-\frac{2n}{1-\frac{n}{N}}  u^2  \right) 
\exp \left (-\frac{n}{4}u^4\right)
\exp\left( -\left[\frac{n^4}{3(N-n)^3}\right]u^4\right)
\]
where
\[
K_{c2}
=
\sqrt{\frac{(1/4)(1-\psi_0)}{ 2\pi\psi_0^2(1-\frac{\mu_0}{2})^2}}
\left(\frac{192}{\mu_0^4e}\right)
\]

\textbf{Case: $k = n$}

When $k=n$ there are only two binomial coefficients to consider in the hypergeometric probability.
Therefore, we must derive a new bound via Stirling's formula. 
Doing so yields
\begin{align}
\frac{{D \choose n}{N-D \choose 0}}{{N \choose n}}
&=
\frac{D!(N-n)!}{(D-n)!N!} \nonumber \\
&\leq 
\sqrt{\frac{ D  (N-n)}{(D-n)N}}
\frac{D^D(N-n)^{(N-n)}}{(D-n)^{(D-n)}N^N}
\exp\left(\frac{1}{12D}+\frac{1}{12(N-n)}-\frac{1}{12(D-n)+1}-\frac{1}{12N+1}\right) \nonumber \\
&=: A'B'C' \ .\nonumber
\end{align}
We can bound $C'$ by

\begin{align}
C' &=
\exp\left(
\frac{12(N-D)+1}{(12D)(12N+1)}
-
\frac{12(N-D)+1}{(12(N-n))(12(D-n)+1)}
\right)\nonumber \\
&=
\exp\left(
\frac{[12(N-D)+1]([(12(N-n))(12(D-n)+1)]-[(12D)(12N+1)])}{[(12D)(12N+1)][(12(N-n))(12(D-n)+1)]}
\right) \nonumber \\
&\leq 1 \nonumber
\end{align}
with the final bound following since $(N-D)>0$ and $[(12(N-n))(12(D-n)+1)] < [(12D)(12N+1)]$.
Continuing with $B'$ we have
\begin{align}
B' 
&= \frac{D^D(N-n)^{(N-n)}}{(D-n)^{(D-n)}N^N} 
     = \frac{\left(\frac{D}{N}\right)^D\left(\frac{N-n}{N}\right)^{(N-D)}}{\left(\frac{D-n}{N-n}\right)^{D-n}} \nonumber \\
&= \left(\frac{\frac{N-n}{N}}{\frac{D-n}{D}}\right)^{D-n}\left(\frac{\frac{N-n}{N}}
                   {\frac{N-n-(D-n)}{N-D}}\right)^{(N-D)}\left(\frac{D}{N}\right)^n
\nonumber \\
&=
\exp\left(-\Gamma + n\log(\mu)\right) \nonumber
\end{align}
where, as before, we have
\[
\Gamma=
(N-n) 
\left[ 
\left(\frac{D-n}{N-n}\right)\log\left(\frac{(D-n)/D}{(N-n)/N}\right)
+ \left(1 -\frac{D-n}{N-n}\right) \log\left(\frac{[N-n - (D-n)]/(N-D)}{(N-n)/N}\right)\right]\ .
\]
Using the previous analysis, we can write
\[
B' = \exp\left(-(N-n)\Psi\left(\frac{f}{\bar{f}}u,1-\mu\right)+n\log(\mu)\right) = 
\exp\left(-(N-n)\Psi\left(\gamma,u\right)+n\log(1-u)\right) 
\]
where we define $\gamma := \frac{f}{\bar{f}}u$, $f:=f_N=\frac{n}{N}$ and $\bar{f}:=\bar{f}_N = 1-f_N = \frac{N-n}{N}$ and use the equality
$u=1-\mu$ under the current hypothesis. Using the analysis from van der Vaart and Wellner, page 461, re-parametrized to the situation at hand, we obtain
\[
\Psi\left(\gamma,u\right) \geq 2\gamma^2 + \gamma^4/3 \ .
\]
We also have the bound via the Taylor expansion:
\[
\log(1-u) = -\left[\sum_{k=1}^\infty \frac{u^k}{k}\right] \leq -\left[\sum_{k=1}^7 \frac{u^k}{k}\right] 
\]

Hence
\begin{align}
B'
&\leq \exp\left(-(N-n)\left[2\left(\frac{f}{\bar{f}}u\right)^2 + \frac{\left(\frac{f}{\bar{f}}u\right)^4}{3}\right]
         +n\log(1-u)\right)  \nonumber \\
&= \exp\left(-2\left(\frac{n^2}{N-n}\right)u^2 -\frac{1}{3}\left(\frac{n^4}{(N-n)^3}\right)u^4 +n\log(1-u)\right)
         \nonumber \\
&\leq \exp\left(-2\left(\frac{n^2}{N-n}\right)u^2 -\frac{1}{3}\left(\frac{n^4}{(N-n)^3}\right)u^4-n\left[\sum_{k=1}^7 \frac{u^k}{k}\right] \right) 
         \nonumber \\
&= \exp\left(- \left(\frac{2nN}{N-n}\right)u^2 -\frac{1}{3}\left(\frac{n^4}{(N-n)^3}\right)u^4 -nu + \frac{3nu^2}{2} 
          - \frac{nu^3}{3}- \frac{nu^6}{6}- \frac{nu^7}{7}\right) \nonumber \\
&\ \cdot \exp\left( - \frac{nu^4}{4}\right)
               \exp\left( - \frac{nu^5}{5}\right) 
               \nonumber \\
&\leq \exp\left(-\frac{2n}{1-\frac{m}{N}}  u^2\right) 
\exp\left(-\frac{1}{3}\left(\frac{n^4}{(N-n)^3}\right)u^4\right)
\exp\left( - \frac{nu^4}{4}\right)
\exp\left( - \frac{nu^5}{5}\right) \ , \nonumber
\end{align}
where the last inequality follows since for $x>0$
\[
x  + \frac{x^3}{3}+ \frac{x^6}{6} + \frac{x^7}{7} - \frac{3}{2}x^2> 0\ .
\]
For $x\ge 0$ this polynomial has a global minimum at $0$ 
and local minimum at $x\approx 0.851662$ with a value of approximately $0.0796078$.
Finally we have
\[
A'
=\sqrt{\frac{ D  (N-n)}{(D-n)N}}
      =\frac{n}{\sqrt{n}}\sqrt{\frac{\frac{D}{N} (1-\frac{n}{N})}{(D-n)\frac{n}{N}}} 
\leq\frac{n}{\sqrt{n}}\sqrt{\frac{(1-\mu_0)(1-\psi_0)}{\psi_0}} 
\]
where the final inequality uses 
the fact that $D-n \geq 1$ in this case. Taking the expression $\exp\left( - \frac{nu^5}{5}\right)$ from 
the bound on $B'$, and observing $u = 1-\frac{D}{N} \geq \mu_0$ we have
\[
n\exp\left( - \frac{nu^5}{5}\right) \leq n\exp\left(-\frac{\mu_0^5}{5}n\right) \leq \frac{5}{\mu_0^5e}
\]
since $xe^{-x} \leq e^{-1}$ for $x > 0$.
Combining the bounds on $A'$,$B'$, and $C'$, we have shown
\[
\frac{{D \choose n}{N-D \choose 0}}{{N \choose n}} 
\leq
\frac{K_{c3}}{\sqrt{n}} 
\exp\left(-\frac{2n}{1-\frac{n}{N}}  u^2\right) 
\exp\left(-\frac{1}{3}\left(\frac{n^4}{(N-n)^3}\right)u^4\right)
\exp\left( - \frac{nu^4}{4}\right)
\]
where
\[
K_{c3} = \sqrt{\frac{(1-\mu_0)(1-\psi_0)}{\psi_0}}
\left(\frac{5}{\mu_0^5e}\right) \ .
\]
Hence if we set $K_{1} = \max(K_{c1},K_{c2},K_{c3})$  we have the bound
\begin{flalign}
  &&
  \frac{{D \choose k}{N-D \choose n-k}}{{N \choose n}} 
  &\leq
  \frac{K_{1}}{\sqrt{n}} \exp\left(-\frac{2n}{1-\frac{n}{N}}  u^2\right) 
  \exp\left(-\frac{1}{3}\left(\frac{n^4}{(N-n)^3}\right)u^4\right) \exp\left(-\frac{nu^4}{4}\right)\ . 
  && 
  \label{proof:intermediate_rep}
\end{flalign}
Plugging in the definitions $k=\sqrt{n}\lambda + n\mu$ and $u = k/n-\mu$
\begin{align}
P\left(\sum_{i=1}^nX_i = k\right)
&=
P\left(\sqrt{n}(\bar{X} - \mu) = \lambda\right) \nonumber \\
&\leq
\frac{K_{1}}{\sqrt{n}} 
\exp\left(-\frac{2\lambda^2}{1-\frac{n}{N}}  \right) 
\exp\left(-\frac{1}{3}\left(\frac{n}{N-n}\right)^3\frac{\lambda^4}{n}\right)
\exp\left( - \frac{\lambda^4}{4n}\right) \ . \nonumber
\end{align}

This gives inequality (i). To obtain inequality (ii), define, for any $n,N$ pair subject to our conditions, 
\[
h(x)=
\left(\frac{2}{1-\frac{n}{N}}\right)  x^2
    + \left ( \frac{1}{3} \left(\frac{n}{N-n}\right)^3 + \frac{1}{4}\right) x^4 
=: ax^2 + bx^4
\]
with $a,b>0$ since $N>n$. Hence $h$ is convex. 
Therefore, as in the Talagrand argument, we also have $h(x) \geq h(u) - (x-u)h'(u)$ for all $x$.
Also for $0\leq x \leq 1$ we see $h'(x)=2ax+4bx^3$ has linear envelopes
\[
2ax \leq h'(x)\leq (2a+4b)x \, .
\]
Let $0 < t < \lambda \leq \sqrt{n}$. Let $k_0 = \lceil n\frac{D}{N} + \sqrt{n}t \rceil = \lceil  n\mu + \sqrt{n}t \rceil$. 
Using the bound at 
\eqref{proof:intermediate_rep} we have
\begin{align}
\sum_{k \geq k_0} \frac{{D \choose k}{N-D \choose n-k}}{{N \choose n}} 
&\leq 
\sum_{k \geq k_0} \frac{K_{1}}{\sqrt{n}} 
\exp\left(-nh(u) -n \left[\frac{k}{n} - \frac{D}{N} - u \right]h'(u) 
\right) \nonumber \\
&= \frac{K_{1}}{\sqrt{n}}\exp\left(-nh(u)\right)
\sum_{k \geq k_0}  
\exp\left(\left[nu - (k - n\mu)\right]h'(u) \right)
\nonumber \\
&\leq
\frac{K_{1}}{\sqrt{n}}\exp\left(-nh(u)\right)
\left[
\frac{
\exp\left(\left[nu - (k_0 - n\mu)\right]h'(u) \right)
}
{1-\exp(-h'(u))}
\right]
\nonumber \\
&\leq
\frac{K_{1}}{\sqrt{n}}\exp\left(-nh(u)\right)
\left[
\frac{K_{ab}}
{h'(u)}
\exp\left(\left[nu - (k_0 - n\mu)\right]h'(u)\right)
\right]
\nonumber \\
&\leq
\frac{K_{1}}{\sqrt{n}}\exp\left(-nh(u)\right)
\left[
\frac{K_{ab}}
{2au}
\exp\left(\left[nu - \sqrt{n}t\right][2a+4b]u\right)
\right]
\nonumber \\
&=
\frac{K_2}{\sqrt{n}u}\exp\left(-nh(u)\right)
\left[
\exp\left(nu\left(u - \frac{t}{\sqrt{n}}\right)[2a+4b]\right)
\right]
\nonumber 
\end{align}
where $K_{ab}$ is a constant that depends on $a$ and $b$, and hence $n$ and $N$, (which we further explain below), and
\[
K_2 = \frac{K_{1}K_{ab}}{2} \ .
\]
We determine $K_{ab}$ by observing $1 - e^{-v} \ge v/M$ for $0 \le v \le v_0$ where
$M = M_{v_0} = v_0/(1-e^{-v_0})$ together with 
\begin{align}
h' (u) 
&\leq  (2a + 4b) u \le 2a + 4b  \ \  \ \ \ \ \ \ \ \ \ \ \ \ \ \ \ \ \ \text{(since }u \le 1\text{)} \nonumber \\
&=  \frac{4}{1 - \frac{n}{N}} + \left ( 1 + \frac{4}{3} \left ( \frac{n/N}{1- \frac{n}{N}} \right )^3 \right ) \equiv v_N \nonumber \\
&\leq  \frac{4}{\psi_0} + \frac{4}{3} \frac{(1-\psi_0)^2}{\psi_0^3} \equiv v_0 . \nonumber
\end{align}
Therefore $K_{a,b}$ can be taken to be
$M = v_0/(1-e^{-v_0} )$ or $M_N = v_N/ (1-e^{-v_N} )$ depending on how much dependence 
on $n$ and $N$ we leave in the bounds. Again by definition we have that
\[
2a+4b =
\left(\frac{4}{1-\frac{n}{N}}\right)+
\left (1+ \frac{4}{3} \left(\frac{n}{N-n}\right)^3\right)
\]  
Therefore we have for all $0 < t < \lambda$ 
\begin{align}
P\left( \sqrt{n} (\overline{X}_n - \mu) \geq t\right) 
&=  \sum_{k \geq k_0} \frac{{D \choose k}{N-D \choose n-k}}{{N \choose n}}  \nonumber \\
&\leq  
\frac{K_2}{\sqrt{n}u}\exp\left(-nh(u)\right)
\exp\left(nu\left(u - \frac{t}{\sqrt{n}}\right)\left[
\left(\frac{4}{1-\frac{n}{N}}\right)+
\left (1+ \frac{4}{3} \left(\frac{n}{N-n}\right)^3\right)
\right]\right) \nonumber \\
&=
\frac{K_2}{\lambda}\exp\left(-nh\left(\frac{\lambda}{\sqrt{n}}\right)\right)
\exp\left(\lambda(\lambda - t)\left[
\left(\frac{4}{1-\frac{n}{N}}\right)+
\left (1+ \frac{4}{3} \left(\frac{n}{N-n}\right)^3\right)
\right]\right) \nonumber
\end{align}
which gives inequality (ii). Inequality (iii) is obtained by setting $t=\lambda$. This completes the proof.
\end{proof}

\textbf{Acknowledgement: } The second author owes thanks to Werner Ehm for several helpful conversations
and to Martin Wells for pointing out the Pitman reference.
We also thank the referee and the editors for a number of constructive queries and suggestions, for pointing
out the recent paper \citep{MR3162712}, 
and for encouraging the inclusion of several 
graphical comparisons of the bounds.

\end{document}